\documentclass[journal]{IEEEtran}
\ifCLASSINFOpdf

\else

\fi

\usepackage[cmex10]{amsmath}
\usepackage[pdftex]{graphicx}
\usepackage{epsfig}
\usepackage{amssymb}
\usepackage{color}
\usepackage{amsthm}
\usepackage{graphics}
\usepackage{empheq}
\usepackage{hyperref}
\usepackage{array}
\usepackage{cite}
\usepackage{bm}
\usepackage{subfig}

\newtheorem{assum}{Assumption}

\newtheorem{prop}{Proposition}

\newtheorem{coro}{Corollary}
\newtheorem{remark}{Remark}
\newtheorem{lem}{Lemma}

\def\f{\frac}

\def\ep{\varepsilon}

\def\i1n{i=1,\cdots,n}
\def\j1n{j=1,\cdots,n}
\def\ij1n{i,j=1,\cdots,n}

\newcommand{\be}{\begin{equation}}
\newcommand{\ee}{\end{equation}}
\newcommand{\beq}{\begin{equation*}}
\newcommand{\eeq}{\end{equation*}}
\newtheorem{thm}{Theorem}

\begin{document}
\setlength\abovedisplayskip{8pt}
\setlength\belowdisplayskip{8pt}
\setlength\abovedisplayshortskip{8pt}
\setlength\belowdisplayshortskip{8pt}
 
\allowdisplaybreaks
 
\setlength{\parindent}{1em}
\setlength{\parskip}{0em}
%
\title{Control and State Estimation of the One-Phase \\Stefan Problem via Backstepping Design}
%
%
%

\author{Shumon~Koga,~\IEEEmembership{Student Member,~IEEE,}
        Mamadou~Diagne,~\IEEEmembership{Member,~IEEE,}
        and~Miroslav~Krstic,~\IEEEmembership{Fellow,~IEEE}
\thanks{S. Koga and M. Krstic are with the Department
of Mechanical and Aerospace Engineering, University of California at San Diego, La Jolla,
CA, 92093-0411 USA (e-mail: skoga@ucsd.edu; krstic@ucsd.edu).}
\thanks{M. Diagne is with the Department of Mechanical Aerospace and Nuclear Engineering of Rensselaer Polytechnic Institute , Troy, NY 12180-3590, USA (email: diagnm@rpi.edu).}}

\maketitle

\begin{abstract}
This paper develops a control and estimation design for the one-phase Stefan problem. The Stefan problem represents a liquid-solid phase transition as time evolution of a temperature profile in a liquid-solid material and its moving interface. This physical process is mathematically formulated as a diffusion partial differential equation (PDE) evolving on a time-varying spatial domain described by an ordinary differential equation (ODE). The state-dependency of the moving interface makes the coupled PDE-ODE system a nonlinear and challenging problem. We propose a full-state feedback control law, an observer design, and the associated output-feedback control law via the backstepping method. The designed observer allows estimation of the temperature profile based on the available measurement of solid phase length. The associated output-feedback controller ensures the global exponential stability of the estimation errors, the ${\cal H}_1$-norm of the distributed temperature, and the moving interface to the desired setpoint under some explicitly given restrictions on the setpoint and observer gain.  The exponential stability results are established considering Neumann and Dirichlet boundary actuations.
\end{abstract}


%
\IEEEpeerreviewmaketitle

\section{Introduction}
\paragraph{Background}
\hspace{10mm}
\IEEEPARstart{S}{tefan} problem,  known as a thermodynamical model of liquid-solid phase transition, has been widely studied since Jeseph Stefan's work in 1889\cite{stefan91}. Typical applications include sea ice melting and freezing \cite{wettlaufer91, maykut71}, continuous casting of steel \cite{petrus12}, crystal-growth \cite{conrad_90},  thermal energy storage system \cite{zalba03}, and lithium-ion batteries \cite{srinivasan04}. For instance, time evolution of the Arctic sea ice thickness and temperature profile was modeled in \cite{maykut71} using the Stefan problem, and the correspondence with the empirical data was investigated. Apart from the thermodynamical model, the Stefan problem has been employed to model a type of population growth, such as in \cite{Friedman1999} for tumor growth process and in \cite{Lei2013} for information diffusion on social networks.

While phase change phenomena described by the  Stefan condition appear in various kinds of science and engineering processes, their mathematical analysis remains quite challenging due to the implicitly given moving interface that reflects the time evolution of a spatial domain, so-called ``free boundary problem" \cite{Gupta03}. Physically, the classical one-phase Stefan problem describes the temperature profile along a liquid-solid material, where the dynamics of the liquid-solid interface is influenced by the heat flux induced by melting or solidification phenomena. Mathematically, the problem  involves a diffusion partial differential equation (PDE) coupled with an ordinary differential equation (ODE). Here, the PDE describes the heat diffusion that provokes melting or solidification of a given material and the ODE  delineates the time-evolution  of the moving front at the solid-liquid interface.

 While the numerical analysis of  the one-phase Stefan problem is broadly covered in the literature, their control related problems have been addressed relatively fewer. In addition to it, most of the proposed control approaches are based on finite dimensional approximations with the assumption of  an explicitly given moving boundary dynamics \cite{Daraoui2010},\cite{Armaou01},\cite{Petit10}. Diffusion-reaction processes with an explicitly known moving boundary dynamics are investigated in  \cite{Armaou01} based on the concept of inertial manifold \cite{Christofides98_Parabolic} and the  partitioning of the infinite dimensional dynamics into slow and fast  finite dimensional modes.     Motion planning boundary control has been adopted in   \cite{Petit10} to ensure asymptotic stability of a one-dimensional one-phase nonlinear Stefan problem   assuming  a prior known  moving boundary  and deriving the manipulated input  from the solutions of   the inverse problem. However, the series representation introduced in   \cite{Petit10}  leads to highly complex solutions that reduce controller design possibilities.

For control objectives, infinite-dimensional frameworks that lead to significant challenges in the process characterization have been developed for the stabilization of  the temperature profile and the moving interface of the Stefan problem. An  enthalpy-based boundary feedback control law that ensures asymptotical stability of the temperature profile and the moving boundary at a desired reference, has been employed in \cite{petrus12}.  Lyapunov analysis is performed in  \cite{maidi2014} based on a  geometric control approach which enables  to adjust the position of a liquid-solid interface at a desired setpoint while  exponentially stabilizing  the ${\cal L}_2$-norm of the distributed temperature. However, the  results in \cite{maidi2014}  are stated based on physical assumptions on the liquid temperature being greater than melting point, which needs to be guaranteed by showing strictly positive boundary input.

Backstepping controller design employs  an invertible transformation that  maps an original
 system into a stable target system. During the past few decades, such a controller design technique has been intensely exploited for the boundary control of diffusion  PDEs defined on a fixed spatial domain as firstly introduced in  \cite{Dejan2001} for the control of a heat equation via measurement of domain-average temperature. For a class of one-dimensional linear parabolic partial integro-differential equations, a state-estimator called backstepping observer was introduced in \cite{Smyshlyaev2005}, which can be combined to the earlier state-feedback boundary controller designed for an identical system \cite{Smyshlyaev2004} to systematically construct output feedback regulators. Over the same period, advection-reaction-diffusion systems with space-dependent thermal conductivity or time-dependent reactivity coefficient were studied in \cite{Smyshlyaev2005b}, and parabolic PDEs containing unknown destabilizing parameters affecting the
interior of the domain or unknown boundary parameters  were adaptively controlled in  \cite{Krstic2008adap,Krstic2008adap2,Krstic2008adap3} combining the backstepping control technique with passive or swapping identifiers. Further advances  on the backstepping control of diffusion equations defined on a multidimensional space or involving in-domain coupled systems can  be found in \cite{Baccoli2015,Deut2016,Vazquez2017,Vazquez2017b}.  

Results devoted to the backstepping stabilization of coupled systems described by a diffusion PDE in cascade with a linear ODE has been primarily presented in \cite{krstic09} with Dirichlet type of boundary interconnection and extended to Neuman boundary interconnection in \cite{susto10, tang11}. For systems relevant with Stefan problem, \cite{Izadi15} designed a backstepping output feedback controller that ensures the exponential stability of an unstable parabolic PDE on \emph{a priori} known dynamics of moving interface which is assumed to be an analytic function in time. Moreover, for PDE-ODE cascaded systems under a state-dependent moving boundary, \cite{Cai2015} developed a local stabilization of nonlinear ODEs with actuator dynamics modeled by wave PDE under time- and state-dependent moving boundary. Such a technique is based on the input delay and wave compensation for nonlinear ODEs designed in \cite{krstic2010, bekiaris2014wave} and its extension to state-dependent input delay compensation for nonlinear ODEs provided by \cite{bekiaris2013compensation}. However, the results in \cite{Cai2015} and \cite{bekiaris2013compensation} that cover state-dependence problems do not ensure global stabilization due to a so-called feasibility condition that needs to be satisfied \emph{a priori}. 

\paragraph{Results and contributions}
Our previous result in \cite{Shumon16} is the first contribution in which global exponential stability of the Stefan problem with an implicitly given moving interface motion is established without imposing any \emph{a priori} given restriction under a state feedback design of Neumann boundary control, and in \cite{Shumon16CDC} we developed the design of an observer-based output feedback control. This paper combines the results of our conference papers \cite{Shumon16} and \cite{Shumon16CDC} and extends them by providing the robustness analysis of the closed-loop system with the state feedback control to the plant parameters' mismatch and the corresponding results under Dirichlet boundary control.

First, a state feedback control law that requires the measurement of the liquid phase temperature profile and the moving interface position is constructed using a novel nonlinear backstepping transformation.  The proposed state feedback controller achieves exponential stabilization of the temperature profile and the moving interface to the desired references in ${\cal H}_1$-norm under the least restrictive condition on the setpoint imposed by energy conservation law.  Robustness of the state feedback controller to thermal diffusivity and latent heat of fusion mismatch is also characterized by explicitly given bounds of the uncertainties' magnitude. Second, an exponentially state observer  which reconstruct the entire distributed temperature profile based solely on the measurement of only the interface position, namely, the length of the melting zone is constructed using the novel backstepping transformation. Finally, combining the state feedback law with the state estimator,  the exponential stabilization of the estimation error, the temperature profile, and the moving interface to the desired references  in the ${\cal H}_{1}$-norm under some explicitly given restrictions on the observer gain and the setpoint.

\paragraph{Organizations}
The one-phase Stefan problem  with Neumann boundary actuation is presented  in Section \ref{model}, and its open loop-stability is discussed in Section \ref{open}. The full-state feedback controller  is constructed  in Section \ref{state-feedback} with a robustness analysis to parameters perturbations. Section \ref{sec:estimation} explains the observer design and in Section \ref{sec:output} is presented the observer-based output-feedback control. In Section \ref{sec:tempcontrol},   both state feedback and output feedback are derived for a Dirichlet type of boundary actuation and robustness analysis is provided for the state feedback case. Simulations to support the theoretical results are given in Section \ref{simu}. The paper ends with final remarks and future directions discussed in  Section \ref{sec:conclusion}.

\paragraph{Notations}
Throughout this paper, partial derivatives and several norms are denoted as 
\begin{align}
u_{t}(x,t) =& \frac{\partial u}{\partial t} (x,t), \quad u_{x}(x,t) = \frac{\partial u}{\partial x} (x,t), \notag\\
||u||_{{\cal L}_2} = &\sqrt{\int_0^{s(t)} u(x,t)^2 dx}, \quad ||u||_{{\cal H}_1} = \sqrt{||u||_{{\cal L}_2}^2 + ||u_{x}||_{{\cal L}_2}^2} \notag
\end{align}

\section*{Part I: Neumann Boundary  actuation}

In this first part, we design a boundary controller for the Stefan problem with a Neumannn boundary actuation.

\section{Description of the Physical Process}\label{model}
\subsection{Phase change in a pure material} 
The one-dimensional Stefan Problem is defined as the  physical model which describes the melting or the solidification mechanism in a pure one-component material of length $L$ in one dimension as depicted  in Fig. \ref{fig:stefan}. The dynamics of the process depends strongly  on the evolution in time of the moving interface (here reduced to a point)  at which phase transition from liquid to solid (or equivalently, in the reverse direction) occurs. Therefore, the melting or solidification mechanism that takes place in  the physical  domain $[0, L]$ induces the existence of   two complementary  time-varying sub-domains, namely,  $[0,s(t)]$ occupied by the liquid phase, and $[s(t),L]$ by the solid phase. Assuming a temperature profile  uniformly equivalent to the  melting temperature in the solid phase,  a dynamical model associated to a melting phenomena  (see Fig. \ref{fig:stefan}) involves only the thermal behavior of  the liquid phase. At a fundamental level, the thermal conduction for a melting component obeys the well known Fourier's Law 
\begin{align}\label{fourier}
q(x,t) = -k T_x (x,t)
\end{align}
where $q(x,t)$ is a heat flux profile, $k$ is the thermal conductivity, and $T(x,t)$ is a temperature profile. Considering a melting material with a density $\rho$ and  heat capacity $C_p$, in the open domain $(0, s(t))$, the energy conservation law  is defined by the following equality
\begin{align}\label{energy}
\rho C_p T_t(x,t) = -q_x(x,t). 
\end{align}
\begin{figure}[t]
\centering
\includegraphics[width=3.2in]{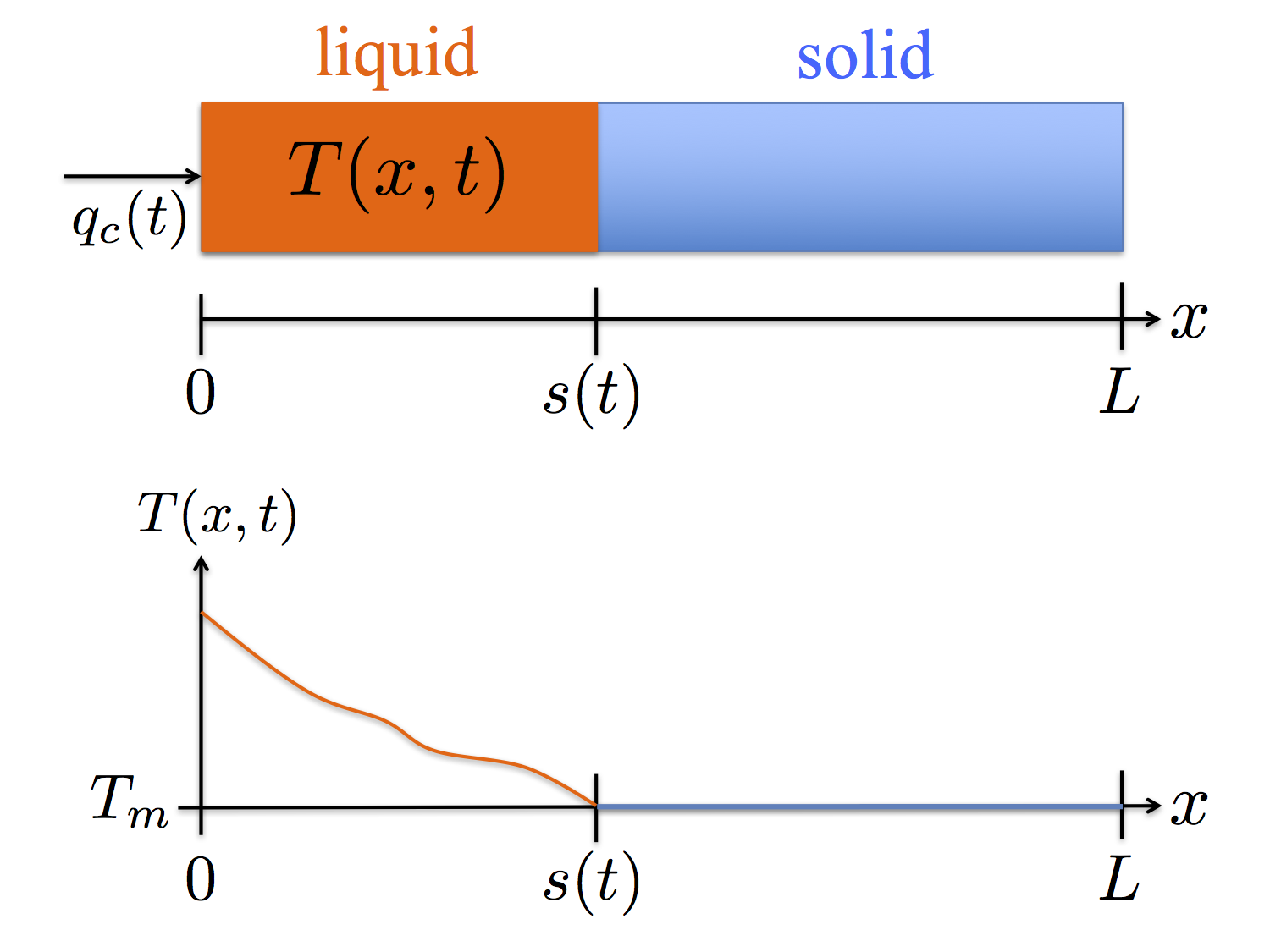}\\
\caption{Schematic of 1D Stefan problem.}
\label{fig:stefan}
\end{figure}
Therefore, a heat flux  entering the system at the boundary at $x=0$ of the liquid phase influences  the dynamics of the solid-liquid interface $s(t)$. Assuming that the temperature in the liquid phase is not lower than the melting temperature  of the material $T_{{\mathrm m}}$ and combining \eqref{fourier} and \eqref{energy}, one can derive the 
 heat equation of  the liquid phase as follows
\begin{align}\label{eq:stefanPDE}
T_t(x,t)=\alpha T_{xx}(x,t), \hspace{1mm} 0\leq x\leq s(t), \hspace{1mm} \alpha :=\f{k}{\rho C_p}, 
\end{align}
with the boundary conditions
\begin{align}\label{eq:stefancontrol}
-k T_x(0,t)=&q_{{\mathrm c}}(t), \\ \label{eq:stefanBC}
T(s(t),t)=&T_{{\mathrm m}},
\end{align}
and the initial conditions
\begin{align}\label{eq:stefanIC}
T(x,0)=T_0(x), \quad s(0) = s_0.
\end{align}
Moreover, for the process described in Fig. \ref{fig:stefan}, 
the local  energy balance at the position of the liquid-solid interface $x=s(t)$ leads to the Stefan condition defined as  the following ODE
\begin{align}\label{eq:stefanODE}
 \dot{s}(t)=-\beta T_x(s(t),t), \quad \beta :=\frac{k}{\rho \Delta H^*},
\end{align}
where $\Delta H^*$ denotes the latent heat of fusion.
Equation \eqref{eq:stefanODE}   expresses the velocity of the liquid-solid moving interface. 

For the sake of brevity, we refer the readers to  \cite{Gupta03}, where the Stefan condition is derived for a solidification process. 
\subsection{Some key properties of the physical model} 

For an  homogeneous melting material, the Stefan Problem presented in Section \ref{model}
 exhibits some important properties  that are stated in the following remarks.

\begin{remark}\emph{
 As the moving interface $s(t)$  governed by \eqref{eq:stefanODE} is unknown explicitly, the problem defined in  \eqref{eq:stefanPDE}--\eqref{eq:stefanODE}  is a nonlinear
problem. Note that this non-linearity is purely geometric
rather than algebraic.}\end{remark}
\begin{remark}\label{assumption}\emph{
Due to the so-called isothermal interface condition
 that prescribes the melting temperature $T_{{\mathrm m}}$ at the interface through  \eqref{eq:stefanBC},
this form of the Stefan problem is a reasonable model only if the following condition holds:
\begin{align}
\label{valid1}T(x,t) \geq& T_{{\mathrm m}} \quad \textrm{ for all }\quad x \in [0,s(t)],
\end{align}}
\end{remark}
The model validity requires the liquid temperature to be greater than the melting temperature and such a condition  yields the following property on moving interface. 
\begin{coro}\label{monoinc}
If the model validity \eqref{valid1} holds, then the moving interface is always increasing, i.e. 
\begin{align}
\label{valid2}\dot{s}(t)\geq&0 \quad \textrm{ for all }\quad t\geq0
\end{align}
\end{coro}
Corollary \ref{monoinc} is established by Hopf's Lemma as shown in \cite{Gupta03}. Hence, it is plausible to impose the following assumption on the initial values. 
\begin{assum}\label{assumini}
The initial position of the moving interface satisfies  $s_0>0$ and the Lipschitz continuity of $T_0(x)$ holds, i.e.  
\begin{align}\label{eq:stefanICbound}
0\leq T_0(x)-T_{{\mathrm m}}\leq H(s_0-x)
\end{align}
\end{assum}
Assumption \ref{assumini} is physically consistent with Remark \ref{assumption}. We recall the following lemma that ensures \eqref{valid1}, \eqref{valid2} for the validity of the model \eqref{eq:stefanPDE}--\eqref{eq:stefanODE}.
\begin{lem}\label{lemma1}
For any $q_{{\mathrm c}}(t)>0$ on the finite time interval $(0,\bar{t})$, the condition of the model validity \eqref{valid1} holds.
\end{lem}
The proof of Lemma \ref{lemma1} is based on Maximum principle as shown in \cite{Gupta03}. In this paper, we focus on this melting problem which imposes the constraint to keep the manipulated heat flux positive. 
\section{Control problem statement and an open-loop stability}\label{open}
\subsection{Problem statement}
The steady-state solution $(T_{{\rm eq}}(x), s_{{\rm eq}})$ of the system \eqref{eq:stefanPDE}-\eqref{eq:stefanODE} with zero manipulating heat flux $q_{c}(t)=0$ yields a uniform melting temperature $T_{{\rm eq}}(x) = T_{{\mathrm m}}$ and a constant interface position given by the initial data. Hence, the system \eqref{eq:stefanPDE}-\eqref{eq:stefanODE} is marginally stable. In this section, we consider the asymptotical stabilization of the interface position $s(t)$ at a desired reference setpoint $s_{{\mathrm r}}$, while the equilibrium temperature profile is maintained  at $T_{{\mathrm m}}$. Thus, the control objective is formulated as
\begin{align}
\lim_{t\to \infty} s(t) &= s_{{\mathrm r}},\label{c1}\\
\lim_{t\to \infty} T(x,t) &= T_{{\mathrm m}}.\label{c2}
\end{align}

\subsection{Setpoint restriction by an energy conservation}
The positivity of the manipulated heat flux in Lemma \ref{lemma1} imposes a restriction on the setpoint due to an energy conservation of the system \eqref{eq:stefanPDE}-\eqref{eq:stefanODE}. The energy conservation law is given by
\begin{align}\label{1stlaw}
\frac{d}{dt}\left(\frac{1}{\alpha}\int_0^{s(t)} (T(x,t)-T_{{\mathrm m}}) dx +\frac{1}{\beta} s(t)\right) = \frac{q_{{\mathrm c}}(t)}{k}. 
\end{align} 
The left hand side of \eqref{1stlaw} denotes the growth of internal energy and its right hand side denotes the external work provided by the injected heat flux. Integrating  the energy balance  \eqref{1stlaw} in $t$ from $0$ to $\infty$ and substituting  \eqref{c1} and \eqref{c2}, the condition to achieve the control objective is given by
 \begin{align}\label{1stlawinteg}
\frac{1}{\beta}(s_{{\mathrm r}}-s_0)-\frac{1}{\alpha}\int_{0}^{s_0} (T_0(x)-T_{{\mathrm m}}) dx = \int_0^{\infty} \frac{q_{{\mathrm c}}(t)}{k} dt .
\end{align}
From relation  \eqref{1stlawinteg}, one can deduce that for any positive heat flux control action which imposes $q_{{\mathrm c}}(t)>0$, the internal energy for a given setpoint must be greater than the initial internal energy. Thus, the following assumption is required.
\begin{assum}\label{assumsetpoint}
The setpoint $s_{{\mathrm r}}$ is chosen to satisfy 
\begin{align}\label{compatibility}
s_{{\mathrm r}}>s_0+\frac{\beta}{\alpha }\int_{0}^{s_0} (T_0(x)-T_{{\mathrm m}}) dx. 
\end{align}
\end{assum}
Therefore, Assumption \ref{assumsetpoint} stands as a least restrictive condition for the choice of setpoint. 
\subsection{Open-loop setpoint control law }\label{sec:pulse}

For any given open-loop  control law $q_c(t)$ satisfying \eqref{1stlawinteg}, the asymptotical stability of system \eqref{eq:stefanPDE}--\eqref{eq:stefanODE} at $s_{\rm r}$ can be  established and the following lemma holds.

\begin{lem}\label{rec*}
Consider an open-loop setpoint control law  $q^{\star}_{{\mathrm c}}(t)$ which satisfies \eqref{1stlawinteg}.  Then, the interface converges asymptotically to the prescribed setpoint $s_{{\mathrm r}}$ and consequently, conditions \eqref{c1} and \eqref{c2} hold. \end{lem}
The proof of Lemma \ref{rec*} can be derived straightforwardly from \eqref{1stlawinteg}.
To illustrate the introduced concept of  open-loop ``energy shaping control" action, we define $\Delta E$ as the difference of total energy given by the left-hand side of \eqref{1stlawinteg}. Hence,
\begin{align}\label{1stlawintegopen}
\Delta E= \frac{1}{\beta}(s_{{\mathrm r}}-s_0)-\frac{1}{\alpha}\int_{0}^{s_0} (T_0(x)-T_{{\mathrm m}}) dx.
\end{align}
For instance, the rectangular pulse control law given by
\begin{align}\label{pulse}
q^{\star}_{{\mathrm c}}(t) = 
\left \{
\begin{array}{cc}
\bar{q} &{\rm for}\quad t\in [0, k\Delta E/\bar{q}]\\
0  &{\rm for}\quad t>k\Delta E/\bar{q}
\end{array}
\right \}
\end{align}
satisfies \eqref{1stlawintegopen} for any choice of the boundary heat flux $\bar{q}$ and thereby, ensures the asymptotical stability of  \eqref{eq:stefanPDE}--\eqref{eq:stefanODE} to the setpoint $(T_{\rm m}, s_{{\mathrm r}})$.

\section{State Feedback Control }\label{state-feedback}

It is remarkable that adopting an open-loop control strategy such as the rectangular pulse \eqref{pulse}, does not allow  to improve the convergence speed. Moreover, the physical parameters of the model need to be   known accurately to avoid robustness issues. More generally, the practical implementation of open-loop control laws in engineering process is limited by performance and robustness issues, thus closed-loop control laws have to be designed to deal with such limitations. 

In this section, we focus on the design of closed-loop backstepping control law for the one-phase Stefan Problem in order to  achieve faster exponential convergence to the desired setpoint $(T_{\rm m}, s_{{\mathrm r}})$ while ensuring the robustness of the closed-loop system to the uncertainty of the physical parameters. We recall that from a physical point of view, for any positive heat flux $q_{{\mathrm c}}(t)$,  the irreversibility of the process restrict \emph{a priori} the choice of the desired setpoint $s_{{\mathrm r}}$ to satisfy \eqref{compatibility}.
\begin{figure}[t]
\centering
\includegraphics[width=3.4in]{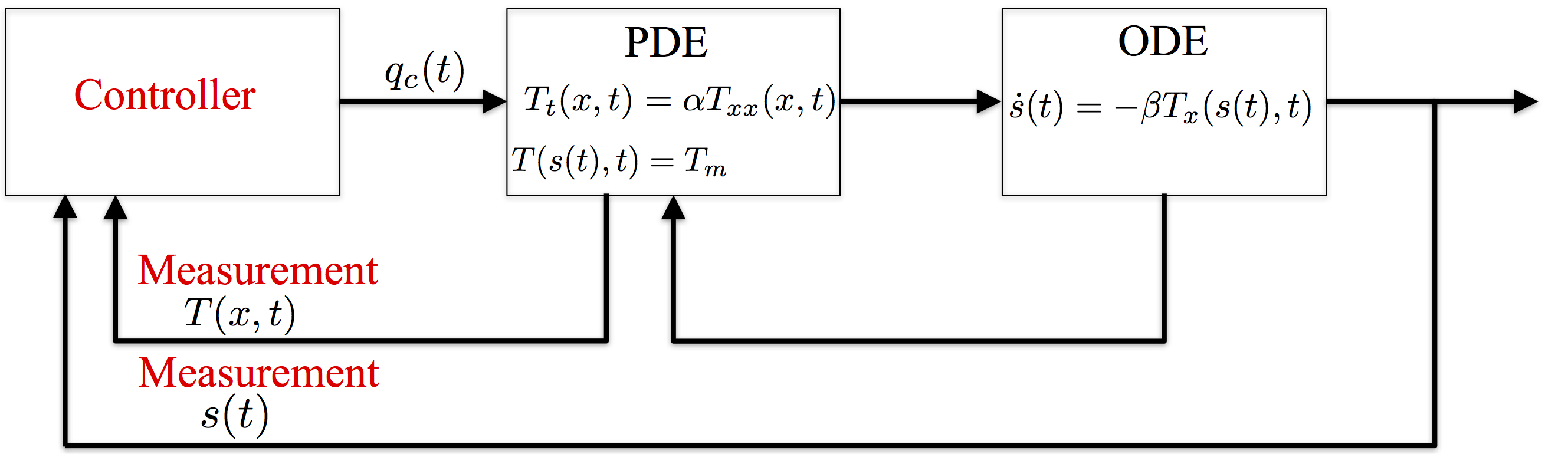}
\caption{Block diagram of the state feedback closed-loop control.}\label{statediag}
\end{figure}

Assuming that the temperature profile in the melt zone $T(x,t)$ and the position of the moving interface $s(t)$ are measured  $\forall x\in [0, s(t)]$ and $\forall t\geq 0$, the following Theorem holds:  
\begin{thm}\label{Theo-1}
Consider a closed-loop system consisting of the plant \eqref{eq:stefanPDE}--\eqref{eq:stefanODE} and the control law
\begin{align}\label{Fullcontrol}
q_{{\mathrm c}}(t)=-ck \left(\frac{1}{\alpha}\int_0^{s(t)} (T(x,t)-T_{{\mathrm m}}) dx +\frac{1 }{\beta}(s(t)-s_{{\mathrm r}})\right), 
\end{align}
where $c>0$ is an arbitrary controller gain. Assume that the initial conditions $(T_0(x), s_0)$ are compatible with the control law and satisfies \eqref{eq:stefanICbound}. Then, for any reference setpoint $s_{{\mathrm r}}$ which satisfies \eqref{compatibility}, the closed-loop system remains the condition of model validity \eqref{valid1} and is exponentially stable in the sense of the norm
\begin{align}\label{H1norm}
||T-T_{{\mathrm m}}||_{{\cal H}_1}^2+(s(t)-s_{{\mathrm r}})^2.
\end{align}
\end{thm}
The proof of Theorem \ref{Theo-1} is established through following steps:
\begin{itemize}
\item A backstepping transformation for moving boundary PDEs and the associated inverse transformation  are constructed for a reference error system  (see Section \ref{fbktarget}).
\item Physical constraints that guarantee the positivity of the boundary heat flux under closed-loop control  are derived (see Section \ref{positivness}).
\item The stability analysis of the target system that induces the stability of the original reference error system is performed (see Section \ref{fbkstability}).
\end{itemize}
\subsection{Backstepping Transformation for Moving Boundary PDEs}\label{fbktarget}
\subsubsection{Reference Error System}
Following  a standard procedure, for  a given reference setpoint $(T_{\rm m}, s_{{\mathrm r}})$, we define the reference errors as 
\begin{align}
u(x,t)=T(x,t)-T_{{\mathrm m}}, \quad X(t)=s(t)-s_{{\mathrm r}}, 
\end{align}
respectively. Then, the reference error system associated to the coupled system \eqref{eq:stefanPDE}--\eqref{eq:stefanODE} is written as
\begin{align}\label{eq:referencePDE}
u_{t}(x,t) =& \alpha u_{xx}(x,t), \quad 0\leq x\leq s(t)\\
 \label{eq:referenceBC1}u_x(0,t) =& -\frac{q_{{\mathrm c}}(t)}{k},\\
 \label{eq:referenceBC2}u(s(t),t) =&0,\\
\label{eq:referenceODE}\dot{X}(t) =&-\beta u_x(s(t),t).
\end{align}

\subsubsection{Direct transformation}
Next, we introduce the following backstepping transformation \footnote{The transformation \eqref{eq:DBST} is an extension of the one initially introduced in \cite{krstic09}, and lately employed in  \cite{susto10, tang11},  to moving boundary problems. }  
\begin{align}\label{eq:DBST}
w(x,t)=&u(x,t)-\frac{\beta}{\alpha} \int_{x}^{s(t)} \phi(x-y)u(y,t) dy\nonumber\\
&-\phi(x-s(t)) X(t), \\
\phi(x) =& \frac{c}{\beta}x\label{ker}, 
\end{align}
which transforms system \eqref{eq:referencePDE}--\eqref{eq:referenceODE} into the following ``target system"
\begin{align}\label{eq:tarPDE}
w_t(x,t)=&\alpha w_{xx}(x,t)+\frac{c}{\beta}\dot{s}(t) X(t), \\
\label{eq:tarBC2}
w_x(0,t) =& 0, \\
\label{eq:tarBC1}
w(s(t),t) =& 0, \\
\label{eq:tarODE}
\dot{X}(t)=&-cX(t)-\beta w_x(s(t),t). 
\end{align}
 The derivation of the explicit gain kernel function \eqref{ker}, which enables to map \eqref{eq:referencePDE}--\eqref{eq:referenceODE} into \eqref{eq:tarPDE}--\eqref{eq:tarODE}, is given in Appendix \ref{appDBST}. Basically, the target system \eqref{eq:tarPDE}-\eqref{eq:tarODE} is obtained by taking the derivatives of \eqref {eq:DBST} with respect to $t$ and $x$ respectively along the solution of \eqref{eq:referencePDE}-\eqref{eq:referenceODE} with the control law \eqref{Fullcontrol}.
\subsubsection{Inverse transformation}
The invertibility of the transformation \eqref{eq:DBST}  guarantees that  the original system \eqref{eq:referencePDE}--\eqref{eq:referenceODE} and the target system \eqref{eq:tarPDE}--\eqref{eq:tarODE} have equivalent stability properties. The inverse transformation of \eqref{eq:DBST} is given by 
\begin{align}\label{Inverse}
u(x,t) = w(x,t) +\frac{\beta}{\alpha} \int_x^{s(t)} \psi(x-y)w(y,t) dy\notag\\
+\psi(x-s(t))X(t),
\end{align}
with an explicit gain kernel function
\begin{align}\label{invback}
\psi(x) = \frac{c}{\beta} \sqrt{\frac{\alpha}{c}} {\rm sin} \left(\sqrt{\frac{c}{\alpha}} x \right). 
\end{align}
As for the direct transformation, the derivation of the inverse tranformation \eqref{Inverse} is  detailed in Appendix \ref{appIBST}. Therefore, there exists a unique inverse transformation and  the stability properties between $(u,X)$-system and $(w,X)$-system are identical. 
\subsection{Physical Constraints}\label{positivness}
As stated in Remark \ref{assumption} and Lemma \ref{lemma1}, a strictly positive heat flux  $q_{{\mathrm c}}(t)$ is required due to the fact that a negative heat flux may lead to a freezing process, which violates the condition of the model validity \eqref{valid1}. To achieve the control objective $s(t) \to s_{{\mathrm r}}$, the aforementioned constraint for a melting Stefan problem does not allow the overshoot beyond the reference setpoint  $s_{{\mathrm r}}$ due to monotonically increasing property of the moving interface dynamics stated in \eqref{valid2}. In this section, we establish that the state feedback control law \eqref{Fullcontrol} guarantees the following ``physical constraints" 
\begin{align}\label{physical constraints}
q_{{\mathrm c}}(t)>&0, \quad  \forall t>0 \\  
\label{constraint2}s(t)<&s_{{\mathrm r}}, \quad  \forall t>0
\end{align}

\begin{prop}
Under Assumption \ref{assumsetpoint}, the closed-loop responses of the plant \eqref{eq:stefanPDE}--\eqref{eq:stefanODE} with the control law \eqref{Fullcontrol} satisfies the physical constraints \eqref{physical constraints} and \eqref{constraint2}, and hence the required conditions for the model validity, namely, \eqref{valid1} and \eqref{valid2} hold. 
\end{prop}
\begin{proof}
By taking the time derivative of \eqref{Fullcontrol} along the solution of \eqref{eq:stefanPDE}--\eqref{eq:stefanODE}, we have
\begin{align}
\dot{q}_{{\mathrm c}}(t)=-cq_{{\mathrm c}}(t)\label{solq}
\end{align}
The differential equation \eqref{solq} leads to the dynamics of the control law to be exponentially decaying function, i.e.  
\begin{align}\label{expcont}
q_{{\mathrm c}}(t)=q_{{\mathrm c}}(0)e^{-ct}. 
\end{align}
Since the setpoint restriction \eqref{compatibility} implies $q_{{\mathrm c}}(0)>0$, the dynamics of the state-feedback control law defined in  \eqref{expcont} verifies condition \eqref{physical constraints}. Then by Lemma 1 and Corollary \ref{monoinc}, the model validity \eqref{valid1} and increasing property on interface \eqref{valid2} are satisfied. Applying \eqref{physical constraints} and \eqref{valid1} to the control law \eqref{Fullcontrol}, the following inequality is obtained
\begin{align}
s(t)<s_{{\mathrm r}},\quad \forall t>0\label{eq:error}.
\end{align}
In addition,  inequality \eqref{valid2}  leads to $s_0< s(t)$. Thus, we have
 \begin{align}\label{position}
s_0<s(t)<s_{{\mathrm r}}, \quad \forall t>0.  
\end{align}
\end{proof}
In the next section,  inequalities  \eqref{valid2} and \eqref{position}  are used to establish the Lyapunov stability of the target system   \eqref{eq:tarPDE}-\eqref{eq:tarODE}. 

\subsection{Stability Analysis}\label{fbkstability}
In the following we establish the exponential stability of  the closed-loop control system in ${\cal H}_1$-norm of the temperature and the moving boundary  based on the analysis of the associated target system \eqref{eq:tarPDE}--\eqref{eq:tarODE}. We consider a functional
\begin{align}\label{eq:lyap1}
V = &\frac{1}{2}||w||_{{\cal H}_1}^2+\frac{p}{2}X(t)^2, 
\end{align}
where $p>0$  is a positive parameter to be determined. Taking the derivative of \eqref{eq:lyap1} along the solution of the target system \eqref{eq:tarPDE}-\eqref{eq:tarODE} and applying Young's, Cauchy-Schwartz, Pointcare's, Agmon's inequalities, with the help of \eqref{valid2} and \eqref{position}, we have
\begin{align}\label{eq:lyapineq}
\dot{V}\leq& -b V  + a\dot{s}(t) V, 
\end{align}
where $a = \max \left\{s_{{\mathrm r}}^2, \frac{8 s_{{\mathrm r}}c}{\alpha}\right\}$, $b =\min\left\{ \frac{\alpha}{4s_{{\mathrm r}}^{2}}, c \right\}$. 
 The detailed derivation of \eqref{eq:lyapineq} is given in Appendix \ref{stabilitytarget}.

However, the second term of the right hand side of \eqref{eq:lyapineq} does not enable to directly conclude the exponential stability. To deal with it, we introduce a new Lyapunov function $W$ such that 
\begin{align}\label{eq:lyap}
W=Ve^{-as(t)}.
\end{align}
The time derivative of \eqref{eq:lyap}  is written as 
\begin{align}\label{dotV0}
\dot{W} = \left(\dot{V} -a\dot{s}(t)V\right)e^{-as(t)},
\end{align}
and the following bound  can be deduced for $\dot{W}$ using  \eqref{eq:lyapineq}
\begin{align}\label{dotV}
\dot{W}  \leq -bW.
\end{align}
Hence, 
 \begin{align}
 W(t) \leq W(0) e^{-bt}, 
 \end{align}
  and using \eqref{eq:lyap} and \eqref{position}, we arrive at 
\begin{align}
V (t)\leq e^{as_{{\mathrm r}}} V(0)  e^{-bt}.
\end{align}
From the definition of $V$ in \eqref{eq:lyap1} the following holds
\begin{align}\label{expstabilityw}
||w||_{{\cal H}_1}^2+pX(t)^2 \leq e^{as_{{\mathrm r}}} \left(||w_0||_{{\cal H}_1}^2+pX(0)^2\right) e^{-bt}. 
\end{align}
Finally,  the direct transformation \eqref{eq:DBST}--\eqref{ker} and its associated  inverse transformation \eqref{Inverse}--\eqref{invback} combined with Young's and Cauchy-Schwarz inequality, enable one to state the existence of a positive constant $D>0$ such that 
\begin{align}\label{expstabilityu}
||u||_{{\cal H}_1}^2+X(t)^2 \leq D \left(||u_0||_{{\cal H}_1}^2+X(0)^2\right) e^{-bt},
\end{align}
which completes the proof of Theorem \ref{Theo-1}. The detailed derivation from \eqref{expstabilityw} to \eqref{expstabilityu} is described in Appendix \ref{stabilityoriginal}. 

Next, we show the robustness of the closed-loop system

\subsection{Robustness to Parameters' Uncertainty}\label{sec:robust}
In this section, we investigate robustness of the backstepping control design \eqref{Fullcontrol} to perturbations on the plant's parameters $\alpha$ and $\beta$. Physically, these perturbations are caused by the uncertainty of the thermal diffusivity and the latent heat of fusion. Hence, we consider the closed-loop system 
\begin{align}\label{eq:robustPDE}
T_t(x,t)=&\alpha (1+\varepsilon_1) T_{xx}(x,t), \hspace{1mm} 0\leq x\leq s(t), \\
\label{eq:robustcontrol}
-k T_x(0,t)=&q_{{\mathrm c}}(t), \\ 
\label{eq:robustBC}T(s(t),t)=&T_{{\mathrm m}},\\
\label{eq:robustODE}
 \dot{s}(t)=&-\beta(1+\varepsilon_2) T_x(s(t),t), 
\end{align}
with the control law \eqref{Fullcontrol}, where $\varepsilon_1$ and $\varepsilon_2$ are perturbation parameters such that $\varepsilon_1>-1$ and $\varepsilon_2>-1$. Defining the   vector $\varepsilon \in {\mathbb R}^2$  where $\varepsilon = (\varepsilon_1, \varepsilon_2)$, the following theorem is stated. 
\begin{thm}\label{robust-thm}
Consider a closed-loop system  \eqref{eq:robustPDE}--\eqref{eq:robustODE} and the control law \eqref{Fullcontrol} under Assumption \ref{assumini} and \ref{assumsetpoint}. Then, for any pair of perturbations such that $\varepsilon \in S_{\varepsilon}$ where 
\begin{align}\label{Se}
S_{\varepsilon} :=& \left\{ (\varepsilon_1, \varepsilon_2) | \left(1-G(c)\right) \varepsilon_1 - G(c)  \leq \varepsilon_2 \leq \varepsilon_1 \right\}, \\
G(c) :=& \left(\frac{3}{10}\right)^{1/4}\frac{ \alpha}{8 s_{{\mathrm r}}^2 c }, 
\end{align}
 the closed-loop system is exponentially stable in the sense of the norm \eqref{H1norm}.
\end{thm}
\begin{proof}
Using the backstepping transformation \eqref{eq:DBST}, the target system  associated to \eqref{eq:robustPDE}--\eqref{eq:robustODE} is defined as  follows
\begin{align}\label{eq:tarrobustPDE}
w_t(x,t)=& \alpha (1+\varepsilon_1)w_{xx}(x,t)+\frac{c}{\beta}\dot{s}(t) X(t)\notag\\
&+ \frac{c }{\beta }\frac{\varepsilon_1-\varepsilon_2}{1+\varepsilon_2} \dot{s}(t) (x-s(t)), \\
\label{eq:tarrobustBC1}w(s(t),t) =& 0, \\
\label{eq:robusDBSu}w_x(0,t) =& 0, \\
\label{eq:tarrobustODE}\dot{X}(t)=&-c(1+\varepsilon_2)X(t)-\beta(1+\varepsilon_2) w_x(s(t),t). 
\end{align}
Next we prove that the control law  \eqref{Fullcontrol}   applied to the perturbed system  \eqref{eq:robustPDE}--\eqref{eq:robustODE}, satisfies the  physical constraints \eqref{physical constraints} and \eqref{constraint2}.
Taking time derivative of \eqref{Fullcontrol} along with \eqref{eq:robustPDE}--\eqref{eq:robustODE}, we arrive at
\begin{align}\label{pertcont}
\dot{q}_{{\mathrm c}}(t)=-c (1+\varepsilon_1)  q_{{\mathrm c}}(t)-c k \left(\varepsilon_1-\varepsilon_2\right)  u_x(s(t),t). 
\end{align}
The positivity of \eqref{Fullcontrol}  is shown using a contradiction argument. Assume that there exists $ t_1>0$ such that $q_{{\mathrm c}}(t)>0$ for $\forall t \in(0,t_1)$ and $q_{{\mathrm c}}(t_1) = 0$. Then, Lemma \ref{lemma1} and Hopf's Lemma leads to $u_x(s(t),t)<0$ for $\forall t \in(0,t_1)$. Since $\varepsilon \in S_{\varepsilon}$ with \eqref{Se}, \eqref{pertcont} implies that
\begin{align}\label{pertineq}
\dot{q}_{{\mathrm c}}(t)> -c (1+\varepsilon_1)  q_{{\mathrm c}}(t), \quad \forall t\in(0,t_1), 
\end{align}
for all $\varepsilon \in S_{\varepsilon}$. Using comparison principle, \eqref{pertineq} and Assumption \ref{assumsetpoint} leads to $q_{{\mathrm c}}(t_1) > q_{{\mathrm c}}(0) e^{-c(1+\varepsilon_1) t_1} >0$. Thus $q_{{\mathrm c}}(t_1)\neq 0$   which is in contradiction with  the assumption $q_{{\mathrm c}}(t_1) = 0$. Consequently,  \eqref{physical constraints} holds by this contradiction argument. Accordingly, \eqref{constraint2} is established using the positivity of the heat flux $q_c(t)$ defined in  \eqref{physical constraints} and the control law \eqref{Fullcontrol}. 

Now, consider a functional
\begin{align}\label{robfuct} 
V_{\ep} = \frac{d}{2} ||w||_{{\cal L}_1}^2 +  \frac{1}{2} ||w_{x}||_{{\cal L}_1}^2 + \frac{p}{2} X(t)^2. 
\end{align}
where the parameters $d$ and $p$ are chosen to be $p = \frac{c\alpha (1+\ep_1)}{8 s_r (1+\ep_2)\beta^2}$, $d =  \frac{160 s_r^2 c^2 (\ep_1-\ep_2)^2}{\alpha^2 (1+\ep_1)^2}$. Taking the time derivative of  \eqref{robfuct} along the solution of \eqref{eq:tarrobustPDE}--\eqref{eq:tarrobustODE}, and applying the aforementioned inequalities in the derivation of \eqref{eq:lyapineq}, we get
\begin{align}\label{robustestimate}
\dot{V}_{\ep} \leq& -d \left(\frac{\alpha (1+\varepsilon_{1})}{4}\right)  ||w_{x}||_{{\cal L}_2}^2\notag\\
&-\frac{\alpha (1+\varepsilon_{1})}{12 }  \left(4 - \left(\frac{\varepsilon_{1}-\varepsilon_{2} }{G(c) (1+\varepsilon_{1})} \right)^4 \right) ||w_{xx}||_{{\cal L}_2}^2 \notag\\
&- \frac{c^2}{\beta^2}\frac{\alpha (1+\varepsilon_{1})}{64 s_{{\mathrm r}}} \left(2 - \left(\frac{\varepsilon_{1}-\varepsilon_{2} }{G(c) (1+\varepsilon_{1})} \right)^4\right)X(t)^2\notag\\
&+\dot{s}(t) \left\{ d^2 s_{{\mathrm r}}^2 || w||_{{\cal L}_2}^2+\frac{ c^2}{\beta ^2}  X(t)^2\right\}. 
\end{align}
From  \eqref{robustestimate} we deduce  that for all $\varepsilon \in S_{\varepsilon}$, there exists positive parameters $a$ and $b$  such that
\begin{align}
\dot{V}_{\ep}\leq& -b V_{\ep} + a\dot{s}(t) V_{\ep}. 
\end{align}
The exponential stability of the target system \eqref{eq:tarrobustPDE}-\eqref{eq:tarrobustODE} can be straightforwardly established following the steps in deriving the proof of exponential  stability of  \eqref{eq:lyap}-\eqref{expstabilityw}, which completes the proof of Theorem \ref{robust-thm}. 
\end{proof}

\section{Sate Estimation Design}\label{sec:estimation}

\subsection{ Problem Statement }\label{statement}
Generally, the implementation of a full-state feedback controller  is rarely achievable  due to the unavailability of full state measurements. The computation of the controller \eqref{Fullcontrol} requires  a full measurement  of the distributed   temperature profile $T(x,t)$ along the domain $[0, s(t)]$  and the moving interface position $s(t)$ which relatively limits its   practical relevance. Here, we extended the full-state feedback results considering  moving interface position $s(t)$ as the only available measurement and deriving an estimator of the temperature profile based on this only available measurement  $Y(t)=s(t)$. Denoting the estimates of the temperature $\hat T(x,t)$, the following theorem holds:
\begin{thm}\label {observer}
Consider the plant \eqref{eq:stefanPDE}--\eqref{eq:stefanODE} with the measurement $Y(t)=s(t)$ and the following observer 
\begin{align}
 \label{observerPDE}\hat{T}_t(x,t)=&\alpha \hat{T}_{xx}(x,t) \nonumber\\
 & - p_1(x,s(t))\left(\frac{\dot{Y}(t)}{\beta} + \hat{T}_x(s(t),t)\right),  \\
 \label{observerBC2}-k\hat{T}_x(0,t)=&q_{{\mathrm c}}(t), \\
\label{observerBC1}\hat{T}(s(t),t)=&T_{{\mathrm m}}, 
\end{align}
where $x\in[0,s(t)]$, and the observer gain $p_1(x,s(t))$ is 
\begin{align}\label{P1x}
p_1(x,s(t)) =  -\lambda s(t)\frac{I_1\left(\sqrt{\frac{\lambda}{\alpha}\left(s(t)^2-x^2\right)}\right)}{\sqrt{\frac{\lambda}{\alpha} \left(s(t)^2-x^2\right)}}
\end{align}
with a gain parameter $\lambda>0 $. Assume that the two physical constraints \eqref{physical constraints} and \eqref{constraint2} are satisfied. Then, for all  $\lambda >0$, the observer error system is exponentially stable in the sense of the norm 
\begin{align}
||T-\hat{T}||_{{\cal H}_1}^2 . 
\end{align}
\end{thm}
Theorem \ref{observer} is proved later in  Section \ref{sec:estimation}. 

\subsection{Observer esign and Backstepping Transformation  }\label{obsvtarget}

\subsubsection{Observer Design and Observer Error System} 

For the reference error  system, namely, the $u$-system  \eqref{eq:referencePDE}--\eqref{eq:referenceODE},  we  consider the following observer:
\begin{align}
\hat{u}_t(x,t)=&\alpha  \hat{u}_{xx}(x,t)\nonumber\\&+p_1(x,s(t))\left(-\frac{1}{\beta}\dot{Y}(t)-\hat{u}_x(s(t),t)\right), \label{obsfirst}\\
\hat{u}(s(t),t)=&0, \\
\label{obslast}\hat{u}_x(0,t)=&-\frac{q_{{\mathrm c}}(t)}{k}, 
\end{align}
where $p_1(x,s(t))$ is the observer gain to be determined. Defining estimation error variable of $u$-system as 
\begin{align}\label{estimateerror}
\tilde{u}(x,t)=u(x,t)-\hat{u}(x,t).
\end{align}
Combining \eqref{eq:referencePDE}--\eqref{eq:referenceODE} with \eqref{obsfirst}--\eqref{obslast},  the $\tilde{u}$-system is written as
\begin{align}\label {eq:errorPDE}
\tilde{u}_t(x,t)=& \alpha \tilde{u}_{xx}(x,t)-p_1(x,s(t))\tilde{u}_x(s(t),t),\\
\label{eq:errorBC2}\tilde{u}(s(t),t)=&0, \quad \tilde{u}_x(0,t)=0.
\end{align}

\subsubsection{Direct transformation}
As for the full state feedback case,   the following backstepping transformation for moving boundary PDEs, whose kernel function is the observer gains 
\begin{align}\label{eq:errorDBST}
\tilde{u}(x,t)=&\tilde{w}(x,t)+\int_x^{s(t)} P_{1}(x,y)\tilde{w}(y,t) dy,
\end{align}
is constructed  to transform the following exponentially stable target system 
\begin{align}\label{wtilde1}
\tilde{w}_t(x,t)=& \alpha \tilde{w}_{xx}(x,t)- \lambda \tilde{w}(x,t),\\
\tilde{w}(s(t),t)=&0, \quad \tilde{w}_x (0,t)=0.\label{wtilde3}
\end{align}
into $\tilde{u}$-system defined  in \eqref{eq:errorPDE}-\eqref{eq:errorBC2}. Taking the derivative of \eqref {eq:errorDBST} with respect to $t$ and $x$ along the solution of \eqref{wtilde1}-\eqref{wtilde3} respectively, the solution of the gain kernel and the observer gain are given by
\begin{align}
\label{p1gain}P_{1}(x,y) =& \frac{\lambda}{\alpha} y\frac{I_1\left(\sqrt{\frac{\lambda}{\alpha} (y^2-x^2)}\right)}{\sqrt{\frac{\lambda}{\alpha} (y^2-x^2)}}, \\
\label{p3gain}p_1(x,s(t))=&-\alpha P_{1}(x,s(t)), 
\end{align}
where $I_1(x)$ is a modified Bessel function of the first kind. 
\subsubsection{Inverse transformation }
Also, one can derive the inverse transformation that maps the $\tilde{w}$-system  \eqref{wtilde1}-\eqref{wtilde3} into the $\tilde{u}$-system  \eqref{eq:errorPDE}-\eqref{eq:errorBC2} as
\begin{align}\label{inverseerror}
\tilde{w}(x,t) =& \tilde{u}(x,t) - \int_{x}^{s(t)} Q_{1}(x,y) \tilde{u}(y,t) dy, \\
Q_{1}(x,y) =& \frac{\lambda}{\alpha} y\frac{J_1\left(\sqrt{\frac{\lambda}{\alpha} (y^2-x^2)}\right)}{\sqrt{\frac{\lambda}{\alpha} (y^2-x^2)}}, \label{bessel}
\end{align}
where $J_1(x)$ is a Bessel function of the first kind. 

 \subsection{Stability Analysis}
To show the stability of the target $\tilde{w}$-system in \eqref{wtilde1}-\eqref{wtilde3}, we consider a functional
\begin{align}\label{eq:lyapunov}
\tilde{V} = &\frac{1}{2} ||\tilde{w}||_{{\cal H}_1}^2.
\end{align}
Taking time derivative of \eqref{eq:lyapunov} along the solution of the target system \eqref{wtilde1}-\eqref{wtilde3}, we obtain 
\begin{align}
\dot{\tilde{V}} = -\alpha|| \tilde{w}_{x}||_{{\cal H}_{1}}^2  - \lambda  ||\tilde{w}||_{{\cal H}_1}^2  -\frac{\dot{s}(t)}{2}\tilde{w}_{x}(s(t),t)^2.
\end{align}
Using  \eqref{physical constraints} and applying Pointcare's inequality, a differential inequality in  $\tilde{V}$ is derived as
\begin{align}
\dot{\tilde{V}}\leq - \left( \frac{\alpha}{4s_{{\mathrm r}}^2} + \lambda\right) \tilde{V}.
\end{align}
Hence, the origin of the target $\tilde{w}$-system is exponentially stable. Since the transformation \eqref{eq:errorDBST} is invertible as in \eqref{inverseerror}, the exponential stability of $\tilde{w}$-system at the origin induces the exponential stability of $\tilde{u}$-system at the origin with the help of \eqref{constraint2}, which completes the proof of Theorem \ref{observer}. 

\section{Observer-Based Output-Feedback Control }\label{sec:output}
An output feedback control law is  constructed  using the reconstruction of the estimated temperature profile through the exponentially convergent observer  \eqref{observerPDE}-\eqref{observerBC1} and the measured interface position $Y(t)=s(t)$  as shown in  Fig. \ref{statediag2} and the following Theorem holds:

\begin{thm}\label{outputthm}
Consider the closed-loop system  \eqref{eq:stefanPDE}--\eqref{eq:stefanODE} with the measurement $Y(t)=s(t)$ and the observer \eqref{observerPDE}-\eqref{observerBC1} under the output feedback control law
\begin{align}\label{outctr}
q_{{\mathrm c}}(t)=&-ck\left(\frac{1}{\alpha} \int_{0}^{Y(t)} \left(\hat{T}(x,t)-T_{{\mathrm m}}\right) dx\right.\notag\\
&\left.+\frac{1 }{\beta} \left(Y(t)-s_{{\mathrm r}}\right) \right). 
\end{align}
Assume that the initial values $\left(\hat{T}_{0}(x), s_0\right)$ are compatible with the control law and the initial plant states $(T_{0}(x), s_0)$ satisfy \eqref{eq:stefanICbound}. Additionally, assume that the upper bound of the initial temperature is known, i.e. the Lipschitz constant $H$ in \eqref{eq:stefanICbound} is known. Then, for any initial temperature estimation $\hat{T}_0(x)$, any gain parameter of the observer $\lambda$, and  any setpoint $s_{{\mathrm r}}$ satisfying \begin{align}\label{restriction1}
T_{{\mathrm m}} +  \hat{H}_{l}(s_0 - x) \leq& \hat{T}_0(x) \leq T_{{\mathrm m}} +  \hat{H}_{u}(s_0 - x), \quad \\
\label{restriction2} \lambda <& \frac{4\alpha}{s_0^2}  \frac{\hat{H}_{l}-H}{\hat{H}_{u}}, \\
\label{restriction3} s_{{\mathrm r}}>&s_0+\frac{\beta s_0^2}{2\alpha } \hat{H}_{u}, 
\end{align}
respectively,  where the parameters $\hat{H}_{u}$ and $\hat{H}_{l}$ satisfy $\label{Hcond}\hat{H}_{u}\geq \hat{H}_{l}>H$, the closed-loop system is exponentially stable in the sense of the norm 
\begin{align}\label{h1output}
&||T-\hat{T}||_{{\cal H}_1}^2+||T-T_{{\mathrm m}}||_{{\cal H}_1}^2+(s(t)-s_{{\mathrm r}})^2 .
\end{align}
\end{thm}
The proof of Theorem \ref{outputthm} is  derived by
\begin{itemize}\item  introducing an appropriate backstepping transformation together with a suitable target system, \item verifying the constraints \eqref{physical constraints} and \eqref{constraint2},
\item and establishing the  Lyapunov stability proof.
\end{itemize}

\begin{figure}[t]
\centering
\includegraphics[width=3.4in]{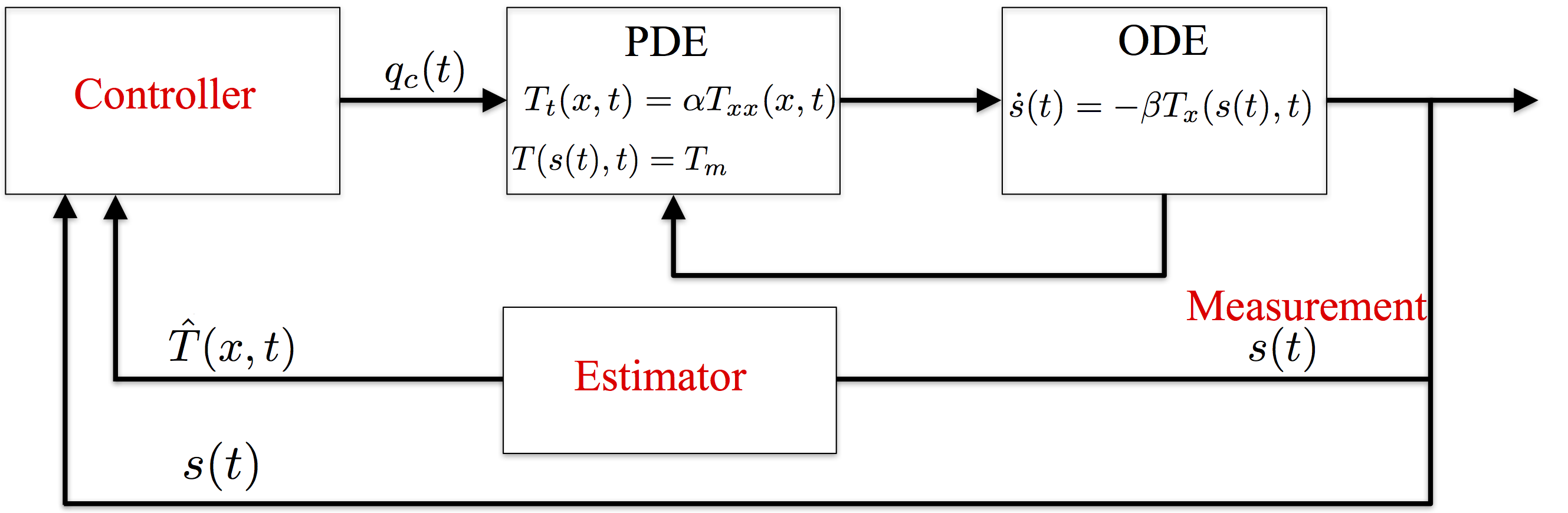}
\caption{Block diagram of observer design and output feedback.}\label{statediag2}
\end{figure}

\subsection{Backstepping Transformation}\label{outputback}
By equivalence, the transformation of the variables $(\hat{u},X)$ into $(\hat{w},X)$ leads to  the gain kernel functions defined by the state-feedback backstepping transformation \eqref{eq:DBST} given by
\begin{align}\label{eq:observerDBST}
\hat{w}(x,t)=\hat{u}(x,t)-\frac{c}{\alpha} \int_{x}^{s(t)} (x-y)\hat{u}(y,t) dy\nonumber\\
-\frac{c}{\beta}(x-s(t)) X(t),
\end{align}
with an associated  target system given by
\begin{align}\label{obsvtarPDE}
\hat{w}_t(x,t)=&\alpha \hat{w}_{xx}(x,t)+\frac{c}{\beta}\dot{s}(t)X(t)\nonumber\\
 &+ f(x,s(t))\tilde{w}_{x}(s(t),t),\\
\hat{w}(s(t),t)=&0, \quad \hat{w}_{x}(0,t)=0,\\
\dot{X}(t)=&-cX(t)-\beta \hat{w}_{x}(s(t),t) - \beta \tilde{w}_{x}(s(t),t) \label{obsvtarODE},
\end{align}
where $f(x,s(t)) = P_1(x,s(t))- \frac{c}{\alpha}  \int_{x}^{s(t)} (x-y) P_1(y,s(t)) dy  -c  (s(t)-x)$.
Evaluating the spatial derivative of \eqref{eq:observerDBST} at  $x=0$, we derive  the output feedback controller  as
\begin{align}\label{output}
q_{{\mathrm c}}(t)&=-ck\left( \frac{1}{\alpha}\int_{0}^{s(t)} \hat{u}(x,t) dx+\frac{1}{\beta} X(t)\right). 
\end{align}
By the same procedure as in Appendix \ref{appA}, an inverse transformation is given by
\begin{align}\label{whatinv}
\hat{u}(x,t)=&\hat{w}(x,t)+ \frac{\beta}{\alpha}\int_{x}^{s(t)} \psi(x-y)\hat{w}(y,t) dy\notag\\
&+ \psi(x-s(t)) X(t), \\
\psi(x) =&\frac{c}{\beta} \sqrt{\frac{\alpha}{c}} {\rm sin} \left(\sqrt{\frac{c}{\alpha}} x \right). 
\end{align}

\subsection{Physical Constraints}\label{sec:constraints}
In this section, we derive sufficient conditions to guarantee the two physical constraints \eqref{physical constraints} and \eqref{constraint2}. First, we state the following lemma. 
\begin{lem}\label{wmaximum}
Suppose that $\tilde{w}(0,t)<0$. Then, the solution of \eqref{wtilde1}-\eqref{wtilde3} satisfies $\tilde{w}(x,t)<0$, $\forall x\in(0,s(t))$, $\forall t>0$.
\end{lem}
The proof of Lemma \ref{wmaximum} is provided by maximum principle as in \cite{nonlinearPDE}.  Next, we state the following lemma. 
\begin{lem}\label{maximum}
For any   initial temperature estimate  $\hat{T}_0(x)$ and, any observer gain parameter $\lambda$  satisfying  \eqref{restriction1} and \eqref{restriction2}, respectively, the following properties hold:
\begin{align}\label{positive}
 &\tilde{u}(x,t) < 0, \quad \tilde{u}_{x}(s(t),t) > 0, \quad \forall x\in(0,s(t)), \quad \forall t>0
\end{align}
\end{lem}
\begin{proof}
Lemma \ref{wmaximum} states that if $\tilde{w}(x,0)<0$, then $\tilde{w}(x,t)<0$. In addition, by the direct transformation \eqref{eq:errorDBST}, $\tilde{w}(x,t)<0$ leads to $\tilde{u}(x,t)<0$ due to the positivity of the solution to the gain kernel \eqref{p1gain}. Therefore, with the help of \eqref{inverseerror}, we deduce that $\tilde{u}(x,t)<0$ if $\tilde{u}(x,0)$ satisfies 
\begin{align}\label{condinv}
 \tilde{u}(x,0) < \int_{x}^{s_0} Q(x,y) \tilde{u}(y,0) dy, \quad \forall x\in(0,s_0). 
\end{align}
Considering the bound of the solution \eqref{bessel} under the condition of \eqref{restriction1}, the sufficient condition for \eqref{condinv} to hold is given by \eqref{restriction2} which restricts the gain parameter $\lambda$. Thus, we have shown that conditions \eqref{restriction1} and \eqref{restriction2} lead to $\tilde{u}(x,t)<0$ for all $ x\in(0,s_0)$. In addition, from the  boundary condition \eqref{eq:errorBC2} and Hopf's lemma, it follows that $\tilde{u}_{x}(s(t),t)>0$. 
\end{proof}

The final step is to prove that the output feedback closed-loop system satisfies the physical constraints \eqref{physical constraints}.  

\begin{prop}\label{proposition}
Suppose the initial values $(\hat{T}_0(x), s_0)$ satisfy \eqref{restriction1} and the setpoint $s_{{\mathrm r}}$ is chosen to satisfy \eqref{restriction3}. 
Then, the physical constraints \eqref{physical constraints} and \eqref{constraint2} are satisfied by  the closed-loop system consisting of the plant   \eqref{eq:stefanPDE}--\eqref{eq:stefanODE}, the observer \eqref{observerPDE}-\eqref{observerBC1} and the output feedback control law \eqref{outctr}.
\end{prop}
\begin{proof}
Taking the time derivative of \eqref{output} along with the solution \eqref{obsfirst}--\eqref{obslast}, with the help of the observer gain \eqref{p3gain},  we derive the following differential equation:
\begin{align}\label{odecontrol}  
&\dot{q}_c(t)=-cq_{{\mathrm c}}(t)+\left( 1+ \int_0^{s(t)} P(x,s(t)) dx \right)\tilde{u}_x(s(t),t).
\end{align}
From the positivity of the solution \eqref{p1gain} and the Neumann boundary value \eqref{positive},  the following differential inequality holds
\begin{align}
&\dot{q}_c(t)\geq -cq_{{\mathrm c}}(t).
\end{align}
Hence, if the initial values satisfy $q_{{\mathrm c}}(0)>0$, equivalently \eqref{restriction3} is satisfied, from     \eqref{output} and \eqref{restriction1}, we get
\begin{align}\label{positivecon}
 q_{{\mathrm c}}(t)>0, \quad \forall t>0. 
 \end{align}
Then, with the relation \eqref{positive} given in Lemma \ref{maximum} and the positivity of $u(x,t)$ provided by Lemma \ref{lemma1}, the following inequality is established :
\begin{align}\label{positivehat}
\hat{u}(x,t) >0, \quad \forall x \in(0,s(t)), \quad  \forall t>0.
\end{align}
Finally, substituting the inequalities \eqref{positivecon} and \eqref{positivehat} into \eqref{output}, we arrive at $X(t) <0$, $\forall t>0$, which guarantees that the second physical constraint \eqref{constraint2} is satisfied. 
\end{proof} 

\subsection{Stability Analysis}\label{outputstability}
We consider a functional
\begin{align}\label{eq:obsvlyapunov}
\hat{V} = &\frac{1}{2} ||\hat{w}||_{{\cal H}_1}^2+\frac{p}{2}X(t)^2+d\tilde{V},
\end{align}
where $d$ is chosen to be large enough and $p$ is appropriately selected. Taking time derivative of \eqref{eq:lyapunov} along the solution of target system \eqref{obsvtarPDE}-\eqref{obsvtarODE}, and applying Young's, Cauchy-Schwarz, Poincare's, Agmon's inequalities, with the help of  \eqref{physical constraints} and \eqref{constraint2}, the following  holds:
\begin{align}\label{estimatelyap}
\dot{\hat{V}}\leq & -b\hat{V} +a\dot{s}(t)\hat{V},
\end{align}
where, $a = {\rm max} \left\{ s_{{\mathrm r}}^2, \frac{16cs_{{\mathrm r}}}{\alpha}\right\}$, $b = {\rm min}\left\{\frac{\alpha}{8s_{{\mathrm r}}^2}, c, 2\lambda\right\}$. Hence, the origin of the $(\hat{w}, X, \tilde{w})$-system is exponentially stable. Since the transformation \eqref{eq:errorDBST} and \eqref{eq:observerDBST} are  invertible as described in \eqref{inverseerror} and \eqref{whatinv},  the exponential stability of $(\hat{w}, X, \tilde{w})$-system at the origin guarantees the exponential stability of $(\hat{u}, X, \tilde{u})$-system at the origin, which completes the proof of Theorem \ref{outputthm}. 
\section*{Part II:  Dirichlet Boundary actuation}
In PART I,  the Neumann boundary actuation of the manipulated heat flux has been considered to design state and output feedback controllers for the one-dimensional Stefan Problem. However, some actuators require Dirichlet boundary control design, such as a thermo-electric cooler actuation controlling a boundary temperature \cite{Boon2014}. In the following section, the boundary temperature is used for the control design. 

\section{Boundary Temperature Control}\label{sec:tempcontrol}
We define the control problem consisting of the same diffusion equation \eqref{eq:stefanPDE}, Dirichlet boundary condition \eqref{eq:stefanBC}, and initial conditions \eqref{eq:stefanIC} : 
 \begin{align}
 \label{PDEtemp}
 T_{t}(x,t) =& \alpha T_{xx}(x,t), \quad 0\le x\le s(t), \\
\label{BCtemp} T(0,t) =& T_{{\mathrm c}}(t), \\
\label{BC2temp}T(s(t),t) =& T_{{\mathrm m}}, \\
\label{ODEtemp}\dot{s}(t) =& - \beta T_{x}(s(t),t)
\end{align}
where $T_{{\mathrm c}}(t)$ is a controlled  temperature. Analogously to Lemma \ref{lemma1}, the following lemma is stated.
\begin{lem}\label{lemma3}
For any $T_{{\mathrm c}}(t)>0$ on the finite time interval $(0,\bar{t})$, $T(x,t)>T_{{\mathrm m}}, ~\forall x\in(0,s(t))$ and $\forall t\in(0,\bar{t})$. And then $\dot{s}(t) >0$, $\forall t\in(0,\bar{t})$. 
\end{lem}
Similarly to Lemma \ref{lemma1}, the proof of Lemma \ref{lemma3} is based on  Maximum principle and Hopf's Lemma \cite{nonlinearPDE}. Therefore, the  following physical constraint holds
\begin{align}\label{tempconst}
T_{c}(t)>0, \quad \forall t>0.
\end{align}
\subsection{Setpoint restriction}
For boundary temperature control, the energy-like conservation law is described as 
\begin{align}\label{tempconserv}
\frac{d}{dt} \left(\frac{1}{\alpha} \int_{0}^{s(t)} x (T(x,t)-T_{{\mathrm m}}) dx + \frac{1}{2\beta} s(t)^2 \right) = T_{c}(t)
\end{align}
Considering the same control objective as in Section \ref{open}, taking the limit of \eqref{tempconserv} from $0$ to $\infty$ yields 
\begin{align}\label{temp1stlawinteg}
\frac{1}{2\beta} (s_{{\mathrm r}}^2 - s_0^2) - \frac{1}{\alpha} \int_0^{s_0} x (T_0(x) - T_{{\mathrm m}}) dx = \int_0^{\infty} T_{c}(t) dt
\end{align}
Hence, by imposing the physical constraint \eqref{tempconst}, the following assumption on the setpoint is required. 
\begin{assum}\label{tempsetpoint}
The setpoint $s_{{\mathrm r}}$ is chosen to satisfy 
\begin{align}\label{tempcompatibility}
s_{{\mathrm r}}>\sqrt{ s_0^2+\frac{2 \beta}{\alpha }\int_{0}^{s_0} x(T_0(x)-T_{{\mathrm m}}) dx}. 
\end{align}
\end{assum}
Again, Assumption \ref{tempsetpoint} stands as a least restrictive condition for the choice of setpoint, and the open-loop stabilization is presented in the following lemma. 
\begin{lem}\label{rec}
Consider an open-loop setpoint control law $T^{\star}_{{\mathrm c}}(t)$ which satisfies \eqref{temp1stlawinteg}.  Then, the control objectives \eqref{c1} and \eqref{c2} are satsified. \end{lem}
Hence, the similar rectangular pulse of energy shaping for temperature control as in Section \ref{sec:pulse} can be derived as 
\begin{align}\label{pulsetemp}
T^{\star}_{{\mathrm c}}(t) = 
\left \{
\begin{array}{cc}
\bar{T} &{\rm for}\quad t\in [0, \Delta E/\bar{T}]\\
0  &{\rm for}\quad t>k\Delta E/\bar{q}
\end{array}
\right \},
\end{align}
 where 
 \begin{align}\label{1stlawintegopentemp}
\Delta E= \frac{1}{2\beta}(s_{{\mathrm r}}^2-s_0^2)-\frac{1}{\alpha}\int_{0}^{s_0} x(T_0(x)-T_{{\mathrm m}}) dx.
\end{align}
\subsection{State feedback controller design}
Firstly, we suppose that the physical parameters are known and  state the following theorem.
\begin{thm}\label{thm3}
Consider a closed-loop system consisting of the plant \eqref{PDEtemp}--\eqref{ODEtemp} and the control law
\begin{align}\label{tempcontrol}
T_{{\mathrm c}}(t)=-c \left(\frac{1}{\alpha} \int_{0}^{s(t)} x\left(T(x,t)-T_{{\mathrm m}}\right) dx\right.\notag \\
\left. +\frac{1}{\beta}s(t) \left(s(t)-s_{{\mathrm r}}\right)\right), 
\end{align}
where $c>0$ is the controller gain. Assume that the initial conditions $(T_0(x), s_0)$ are compatible with the control law and satisfies \eqref{eq:stefanICbound}. Then, for any reference setpoint  $s_{{\mathrm r}} $ and control gain $c$ which satisfy 
\begin{align}\label{temprestriction}
s_{{\mathrm r}} >& s_0 + \frac{\beta}{\alpha} \int_{0}^{s_0} \frac{x}{s_0}u_0(x) dx, \\
\label{tempgainrest}c \leq& \frac{\alpha}{2\sqrt{2}s_{{\mathrm r}} }, 
\end{align}
respectively, the closed-loop system is exponentially stable in the sense of the norm \eqref{H1norm}.
\end{thm}
\begin{proof}
The backstepping  transformation \eqref{eq:DBST} leads to the following target system \begin{align}\label{tarTPDE}
w_t(x,t)=&\alpha w_{xx}(x,t)+\frac{c}{\beta}\dot{s}(t) X(t), \\
\label{tarTBC1}
w(s(t),t) =& 0,\\
\label{tarTBC2}w(0,t) =& 0, \\
\label{tarTODE}\dot{X}(t)=&-cX(t)-\beta w_x(s(t),t)
\end{align}
and the control law \eqref{tempcontrol}. Next, we show that the physical constraints \eqref{tempconst} and \eqref{constraint2} are insured if \eqref{temprestriction} holds. 
Taking a time derivative of \eqref{tempcontrol}, we have 
\begin{align}\label{tempdif}
\dot{T}_{{\mathrm c}}(t)=-c T_{{\mathrm c}}(t)-\frac{c}{\beta}\dot{s}(t)X(t)
\end{align}
Assume that $\exists t_2$ such that $T_{{\mathrm c}}(t)>0$ for $\forall t\in(0,t_2)$ and $T_{{\mathrm c}}(t_2) = 0$. Then, by maximum principle, we get  $u(x,t)>0$ and $\dot{s}(t)>0$ for $\forall t\in(0,t_2)$. Hence, $s(t)>s_0>0$. From \eqref{tempcontrol}, the following equality is deduced 
\begin{align}
X(t)=-\frac{\beta}{cs(t)}\left(T_{{\mathrm c}}(t)+\frac{c}{\alpha} \int_{0}^{s(t)} xu(x,t) dx\right).
\end{align}
Thus, $X(t)<0$ for $\forall t\in(0,t_2)$ and  \eqref{tempdif} verifies a differential inequality 
\begin{align}
\dot{T}_{{\mathrm c}}(t)> -c T_{{\mathrm c}}(t), \quad \forall t\in(0,t_2).
\end{align}
Comparison principle and \eqref{temprestriction} yield $T_{{\mathrm c}}(t_2) > T_{{\mathrm c}}(0) e^{-c t_2}>0$ in contradiction to  $T_{{\mathrm c}}(t_2) = 0$. Therefore, $\nexists t_2$ such that $T_{{\mathrm c}}(t)>0$ for $\forall t\in(0,t_2)$ and $T_{{\mathrm c}}(t_2) = 0$, which implies $T_{{\mathrm c}}(t)>0$ for $\forall t>0$ assuming \eqref{temprestriction}. 
Finally, we consider a functional 
\begin{align}\label{V3}
V = \frac{d}{2}||w||_{{\cal L}_2}^2 + \frac{1}{2}||w_{x}||_{{\cal L}_2}^2 + \frac{p}{2} X(t)^2. 
\end{align}
With an appropriate choice of positive parameters $d$ and $p$, time derivative of \eqref{V3} yields 
\begin{align}
\dot{V} \leq& - \left(\frac{\alpha}{2} -\sqrt{2} cs_{{\mathrm r}} \right) || w_{xx}||^2- \frac{d\alpha}{2(4s_{{\mathrm r}}^2 + 1)}  ||w||_{{\cal H}_1}^2 \notag\\
&- \frac{\alpha c^2}{4\beta^2} X(t)^2+\dot{s}(t) \left(\frac{c^2}{\beta^2} X(t)^2 + \frac{d^2 s_{{\mathrm r}}^2}{2}||w||^2 \right). 
\end{align}
Therefore, choosing the controller gain to satisfy \eqref{tempgainrest}, it can be verified that there exists positive parameters $b$ and $a$ such that 
\begin{align}
\dot{V} \leq &- b V + a \dot{s}(t) V. 
\end{align}
Similarly, in the Neumann boundary actuation case, under the physical constraint \eqref{tempconst}, the exponential stability of the target system \eqref{tarTPDE}-\eqref{tarTODE} can be established, which completes the proof of Theorem \ref{thm3}. 
\end{proof}

\subsection{Robustness to parameters' uncertainty}
Next, we investigate robustness of the boundary temperature controller  \eqref{tempcontrol} to perturbations on the plant's physical parameters $\alpha $ and $\beta$. Again, the perturbed system is described as 
\begin{align}
 \label{PDEtemp-robust}
 T_{t}(x,t) =& \alpha(1+\ep_1) T_{xx}(x,t), \quad 0\le x\le s(t), \\
\label{BCtemp-robust} T(0,t) =& T_{{\mathrm c}}(t), \\
\label{BC2temp-robust}T(s(t),t) =& T_{{\mathrm m}}, \\
\label{ODEtemp-robust}\dot{s}(t) =& - \beta(1+\ep_2) T_{x}(s(t),t)
\end{align}
where $\varepsilon_1$ and $\varepsilon_2$ are perturbation parameters such that $\varepsilon_1>-1$ and $\varepsilon_2>-1$.
\begin{thm}\label{robust-temp}
Consider a closed-loop system \eqref{PDEtemp-robust}--\eqref{ODEtemp-robust}, and the control law \eqref{tempcontrol}. Then, for any perturbations $(\ep_1, \ep_2)$ which satisfy $\ep_1\ge \ep_2$, there exists $c^*$ such that for $0<\forall c<c^*$ the closed-loop system is exponentially stable in the norm \eqref{H1norm}.
\end{thm}
\begin{proof}
Note that  the transformation \eqref{eq:DBST}--\eqref{ker}  and the system \eqref{eq:robustPDE}, \eqref{eq:robustBC}-\eqref{eq:robustODE} are identical to the ones considered in   Section \ref{sec:robust}. Moreover, only the boundary condition of the target system \eqref{eq:tarrobustPDE}--\eqref{eq:tarrobustODE} at $x=0$ is modified as $w(0,t) = 0$, in order to  match the temperature control problem.

Condition \eqref{constraint2} and \eqref{tempconst} need to be satisfied under  the parameter perturbations. Taking time derivative of \eqref{tempcontrol} along the system  \eqref{eq:robustPDE}, \eqref{eq:robustBC}-\eqref{eq:robustODE}, with the boundary condition \eqref{BCtemp}, we obtain 
\begin{align}
\dot{T}_c(t)=&-c (1+\varepsilon_1)  T_c(t) - \frac{c}{\beta} \dot{s}(t)X(t)\notag\\
&-c (\varepsilon_1 - \varepsilon_2) u_x(s(t),t).
\end{align}
The inequality  $\varepsilon_1>\varepsilon_2$ enables to state  the positivity of the controller $T_{c}(t)>0$. Hence, the physical constraints \eqref{tempconst} and \eqref{constraint2} are verified.
Finally, we introduce the following  functional
\begin{align}
V_{\ep} = \frac{d}{2} ||w||_{{\cal L}_1}^2 +  \frac{1}{2} ||w_{x}||_{{\cal L}_1}^2 + \frac{p}{2} X(t)^2. 
\end{align}
With an appropriate choice of $d$ and $p$, we have 
\begin{align}\label{robustestimate2}
\dot{V}_{\ep} \leq& - \frac{d\alpha (1+\ep_1)}{4}  ||w_{x}||_{{\cal L}_2}^2\notag\\
&- \frac{\alpha (1+\ep_1)}{8}  \left( 2 - Ac^3 \right) ||w_{xx}||_{{\cal L}_2}^2 \notag\\
&- \frac{c^2 \alpha (1+\ep_1)}{32\beta^2 s_r} \left( 2 - Ac^3 - Bc  \right)X(t)^2\notag\\
&+\dot{s}(t) \left\{ d^2 s_{{\mathrm r}}^2 || w||_{{\cal L}_2}^2+\frac{ c^2}{\beta ^2}  X(t)^2\right\}.
\end{align}
where $A = \frac{2^9\sqrt{2} s_r^6 (1+s_r) (\ep_1-\ep_2)^2}{3 \alpha^3 (1+\ep_1)^2 (1+\ep_2)}$, $B = \frac{16\sqrt{2} s_r^2 }{\alpha (1+\ep_2)}$. 
Let $c^*$ be a positive root of $A c^{*3}+Bc^*=1$. Then, for $0<\forall c<c^*$, there exists positive parameters $a$ and $b$ which verifies 
\begin{align}
\dot{V}_{\ep}\leq& -b V_{\ep} + a\dot{s}(t) V_{\ep}. 
\end{align}
Hence, it concludes Theorem \ref{robust-temp}. 
\end{proof}

\subsection{Observer Design with Boundary Temperature Controller}
With respect to the boundary temperature control introduced in Section \ref{sec:tempcontrol} instead of the heat control, the observer design is replaced by the following. 
\begin{coro}\label{obsvtemp}
Consider the following  closed-loop system  of the observer 
\begin{align}
 \label{observerPDEtemp}\hat{T}_t(x,t)=&\alpha \hat{T}_{xx}(x,t) \nonumber\\
 & - p_2(x,s(t))\left(\frac{\dot{Y}(t)}{\beta} + \hat{T}_x(s(t),t)\right),  \\
 \label{observerBC2temp}\hat{T}(0,t)=&T_{{\mathrm c}}(t), \\
\label{observerBC1temp}\hat{T}(s(t),t)=&T_{{\mathrm m}}, 
\end{align}
where $x\in[0,s(t)]$, and the observer gain $p_2(x,s(t))$ is 
\begin{align}\label{P2x}
p_2(x,s(t)) =  -\lambda x\frac{I_1\left(\sqrt{\frac{\lambda}{\alpha}\left(s(t)^2-x^2\right)}\right)}{\sqrt{\frac{\lambda}{\alpha} \left(s(t)^2-x^2\right)}}
\end{align}
with an observer gain $\lambda>0 $. Assume that the two physical constraints \eqref{tempconst} and \eqref{constraint2} are satisfied. Then, for all  $\lambda >0$, the observer error system is exponentially stable in the sense of the norm 
\begin{align}
||T-\hat{T}||_{{\cal H}_1}^2 . 
\end{align}
\end{coro}
The proof of Corollary \ref{obsvtemp} is established by the same procedure as in Section \ref{obsvtarget}.

\subsection{Output feedback controller design}
For boundary temperature controller \eqref{tempcontrol}, the output feedback controller is designed using the state observer \eqref{observerPDEtemp}--\eqref{P2x} presented in Corollary \ref{obsvtemp}. The following corollary is stated for a controlled boundary temperature. 
\begin{coro}\label{outputthmtemp}
Consider the closed-loop system consisting of the plant \eqref{PDEtemp}--\eqref{ODEtemp}, the measurement $Y(t)=s(t)$, the observer \eqref{observerPDEtemp}-\eqref{observerBC1temp}, and the output feedback control law
\begin{align}\label{outctrtemp}
T_{{\mathrm c}}(t)=-c \left(\frac{1}{\alpha} \int_{0}^{Y(t)} x\left(\hat{T}(x,t)-T_{{\mathrm m}}\right) dx\right.\notag \\
\left. +\frac{1}{\beta}Y(t) \left(Y(t)-s_{{\mathrm r}}\right)\right). 
\end{align}
With  $c$, $\hat{T}_0(x)$, $\lambda$ satisfying \eqref{tempgainrest}, \eqref{restriction1}, and \eqref{restriction2}, respectively, and, the setpoint $s_{{\mathrm r}}$ satisfying  
\begin{align}
\label{restriction32} s_{{\mathrm r}}>&s_0+\frac{\beta s_0^2}{6\alpha } \hat{H}_{u}, 
\end{align}
the closed-loop system is exponentially stable in the sense of the norm \eqref{h1output}.
\end{coro}
The proof of Corollary \ref{outputthmtemp} can be established following the steps of the proof Corollary \ref{obsvtemp} detailed in Section \ref{sec:output}.
\section{Numerical Simulation}\label{simu}
Simulation results  in the case of the Neumann boundary actuation developed  in Part I are performed considering a strip of zinc as in \cite{maidi2014} whose physical properties are given in Table 1 (The consistent simulations to  illustrate  the feasibility of the backstepping controller with Dirichlet boundary actuation are easily achievable  but due to space limitation, these are not provided). Here, we use the well known boundary immobilization method combined with finite difference semi-discretization. The setpoint and the initial values are chosen as $s_{{\mathrm r}}$ = 0.35 m, $s_0$ = 0.01 m, and $T_0(x)-T_{{\mathrm m}}=H(s_{0}-x)$ with $H$ = 10000 K$\cdot {\rm m}^{-1}$. Then, the setpint restriction \eqref{compatibility} is satisfied. 

\subsection{State Feedback Control and its Robustness}
\subsubsection{Comparison of the pulse input and the backstepping control law}
Fig. \ref{fig:Robust1} shows the responses of the plant \eqref{eq:robustPDE}--\eqref{eq:robustODE} with the  open-loop pulse input \eqref{pulse} (dashed line) and the backstepping control law \eqref{Fullcontrol} (solid line). The time window of the open-loop pulse input is set to be 50 min. The gain of the backstepping control law is chosen as $c$=0.001 to be small enough to avoid numerical instabilities. Fig. \ref{fig:Robust1} (a) shows the response of $s(t)$ without the parameters perturbations, i.e. $(\varepsilon_1, \varepsilon_2)=(0,0)$ and clearly demonstrates that $s(t)$ converges to $s_{{\mathrm r}}$ applying both rectangular  pulse input and backstepping control law, and the convergence speed with the backstepping control is faster. Moreover, from the dynamics of $s(t)$ under parameters perturbations $(\varepsilon_1, \varepsilon_2)=(0.3,-0.2)$ shown in Fig. \ref{fig:Robust1} (b), it can be seen that the convergence of $s(t)$ to $s_{{\mathrm r}}$ is successfully achieved  with the backstepping control law but  not obtained  with the pulse input. On both Fig. \ref{fig:Robust1} (a) and (b), the responses with the backstepping control law show that the interface position  converge faster without the overshoot beyond the setpoint, i.e. $\dot{s}(t)>0$ and $s_0<s(t)<s_{{\mathrm r}}$ for $\forall t>0$. 

\subsubsection{Validity of the physical constraints}
The dynamics  of the controller $q_{{\mathrm c}}(t)$ and the temperature at the initial interface $T(s_0,t)$ with the backstepping control law \eqref{Fullcontrol} are described in Fig. \ref{fig:Robust2} (a) and (b), respectively,  for the system without  parameter's uncertainties, i.e., $(\varepsilon_1, \varepsilon_2)=(0,0)$ (red) and the system with parameters' mismatch  $(\varepsilon_1, \varepsilon_2)=(0.3,-0.2)$ (blue). As presented in Fig. \ref{fig:Robust2} (a), the boundary heat controller $q_{{\mathrm c}}(t)$ remains positive, i.e. $q_{{\mathrm c}}(t)>0$ in both cases. Moreover,  Fig. \ref{fig:Robust2} (b) shows that $T(s_0,t)$ converges to $T_{{\mathrm m}}$ with $T(s_0,t)>T_{{\mathrm m}}$ for both the system with accurate parameters and the system with uncertainties on the parameters. Physically,   Fig. \ref{fig:Robust2} (b) means that the temperature at the initial interface warms up from the melting temperature $T_{{\mathrm m}}$, to enable  melting  the solid-phase to the setpoint $s_{{\mathrm r}}$ before settling back  to $T_{{\mathrm m}}$. These physical constraints and phenomena holds even with the parameter uncertainty as long as  \eqref{Se} is satisfied. Therefore, the numerical results are consistent with our theoretical result. 

	\begin{table}[t]
\caption{Physical properties of zinc}
\begin{center}
    \begin{tabular}{| l | l | l | }
    \hline
    $\textbf{Description}$ & $\textbf{Symbol}$ & $\textbf{Value}$ \\ \hline
    Density & $\rho$ & 6570 ${\rm kg}\cdot {\rm m}^{-3}$\\ 
    Latent heat of fusion & $\Delta H^*$ & 111,961${\rm J}\cdot {\rm kg}^{-1}$ \\ 
    Heat Capacity & $C_p$ & 389.5687 ${\rm J} \cdot {\rm kg}^{-1}\cdot {\rm K}^{-1}$  \\  
    Thermal conductivity & $k$ & 116 ${\rm w}\cdot {\rm m}^{-1}$  \\ \hline
    \end{tabular}
\end{center}
\end{table}

\subsection{Observer Design and Output Feedback Control}

The initial estimation of the temperature profile is set to $\hat{T}_0(x)-T_{{\mathrm m}}=\hat{H}(s_{0}-x)$ with $\hat{H}$ = 1000 K$\cdot {\rm m}^{-1}$, and the observer gain  is chosen as $\lambda$ = 0.001. Then, the restriction on $\hat{T}_0(x)$, $\lambda$, and $s_{{\mathrm r}}$ described in \eqref{restriction1}-\eqref{restriction3} are satisfied, which are conditions for Theorem \ref{observer} and Theorem \ref{outputthm} to remain valid. 

The dynamics of the moving interface $s(t)$, the output feedback controller $q_{{\mathrm c}}(t)$, and the temperature at the initial interface $T(s_0,t)$ are depicted in Fig. \ref{fig:simulation} (a)--(c), respectively. Fig. \ref{fig:simulation} (a) shows that the interface $s(t)$ converges to the setpoint $s_{{\mathrm r}}$ without overshoot which is guaranteed in Proposition \ref{proposition}.  Fig. \ref{fig:simulation} (b) shows that the output feedback controller remains positive, which is a physical constraint for the model to be valid as stated in Lemma \ref{lemma1} and ensured in Proposition \ref{proposition}. The model validity can be seen in Fig. \ref{fig:simulation} (c) which illustrates $T(s_0,t)$ increases from the melting temperature $T_{{\mathrm m}}$  to enable melting of material and settles back to its equilibrium. The positivity of the backstepping  controller shown in Fig. \ref{fig:simulation} (b), results from the negativity of the distributed estimation error $\tilde{T}(x,t)$ as shown in Lemma \ref{maximum} and Proposition \ref{proposition}. Fig. \ref{fig:simulation} (d) shows the dynamics of estimation errors of distributed temperature at $x=0$ (red), $x=s(t)/4$ (blue), and $x=s(t)/2$ (green), respectively. It is remarkable that the estimation errors at each point converge to zero and remains negative confirming  the theoretical results stated in Lemma \ref{maximum} and Proposition \ref{proposition}. 

\begin{figure*}[!ht]
\centering
\subfloat[with accurate parameters $(\epsilon_1, \epsilon_2) = (0,0)$ ]{\includegraphics[width=3.0in]{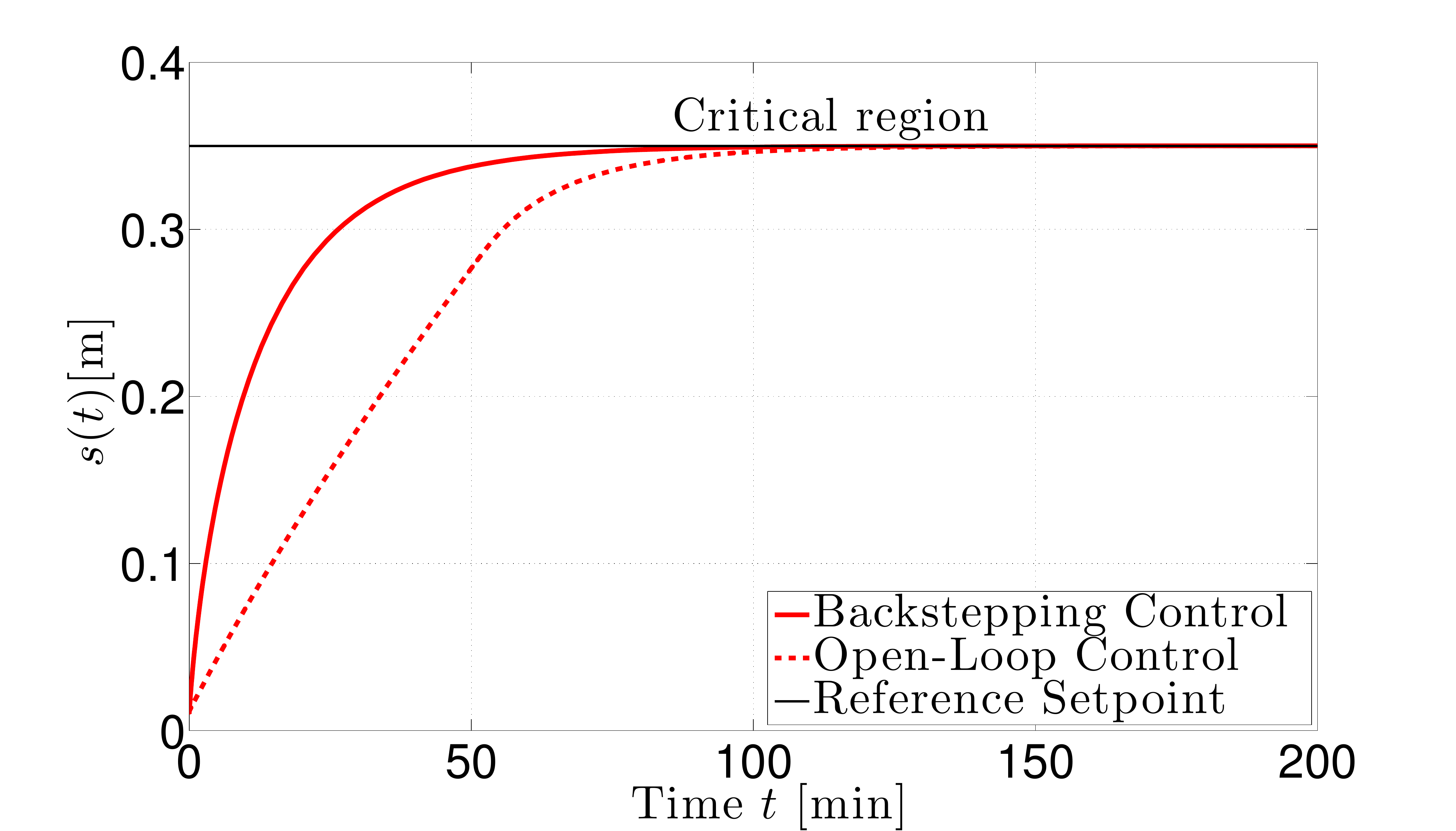}}\label{6a}
\subfloat[under parameters perturbation $(\epsilon_1, \epsilon_2) = (0.3,-0.2)$]{\includegraphics[width=3.0in]{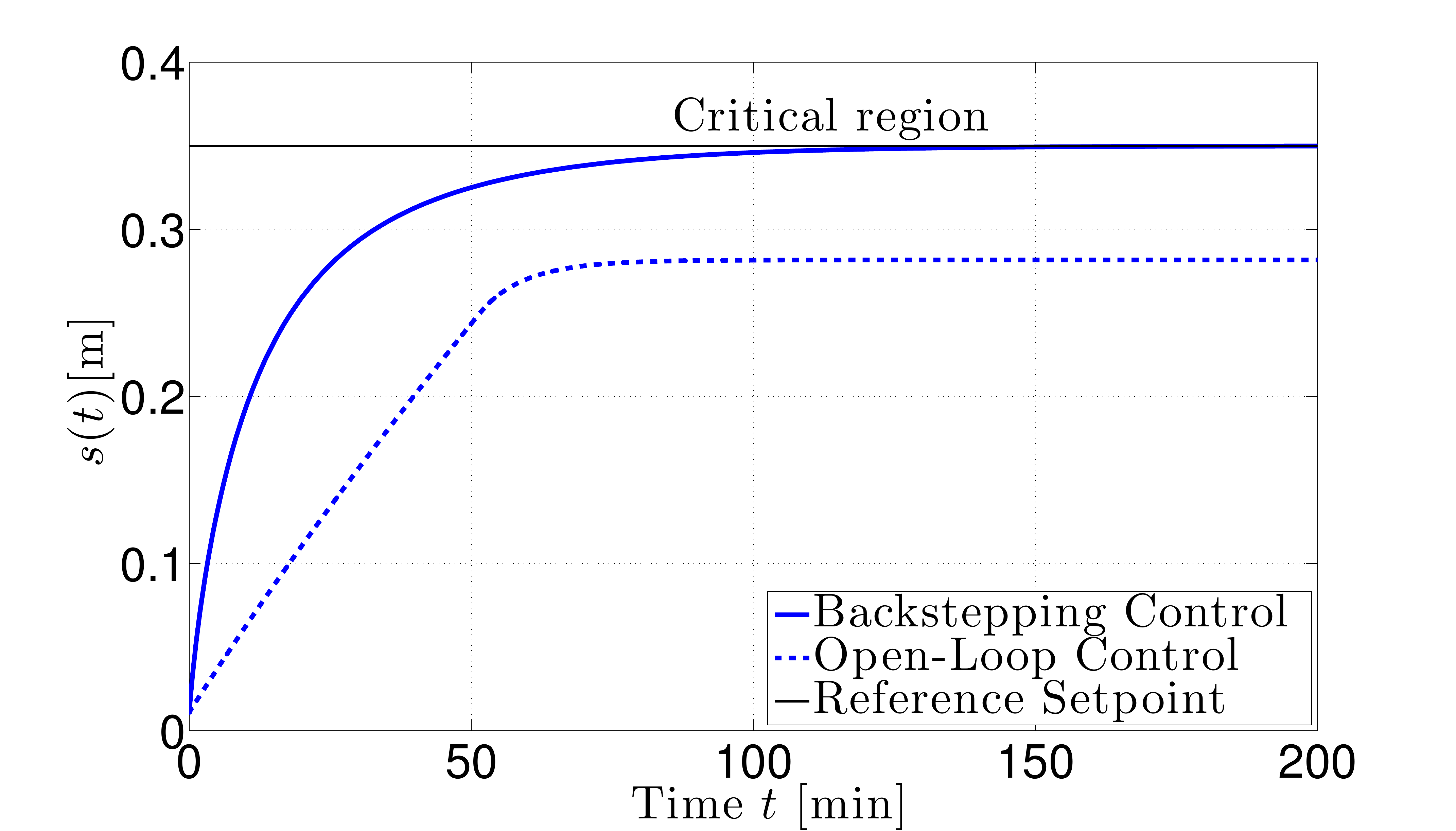}}
\caption{The moving interface responses of the plant \eqref{eq:robustPDE}--\eqref{eq:robustODE} with the open-loop pulse input \eqref{pulse}  (dashed line) and the backstepping control law \eqref{Fullcontrol} (solid line).}
\label{fig:Robust1}
\subfloat[Positivity of the controller remains, i.e. $q_{c}(t)>0$.]{\includegraphics[width=3.0in]{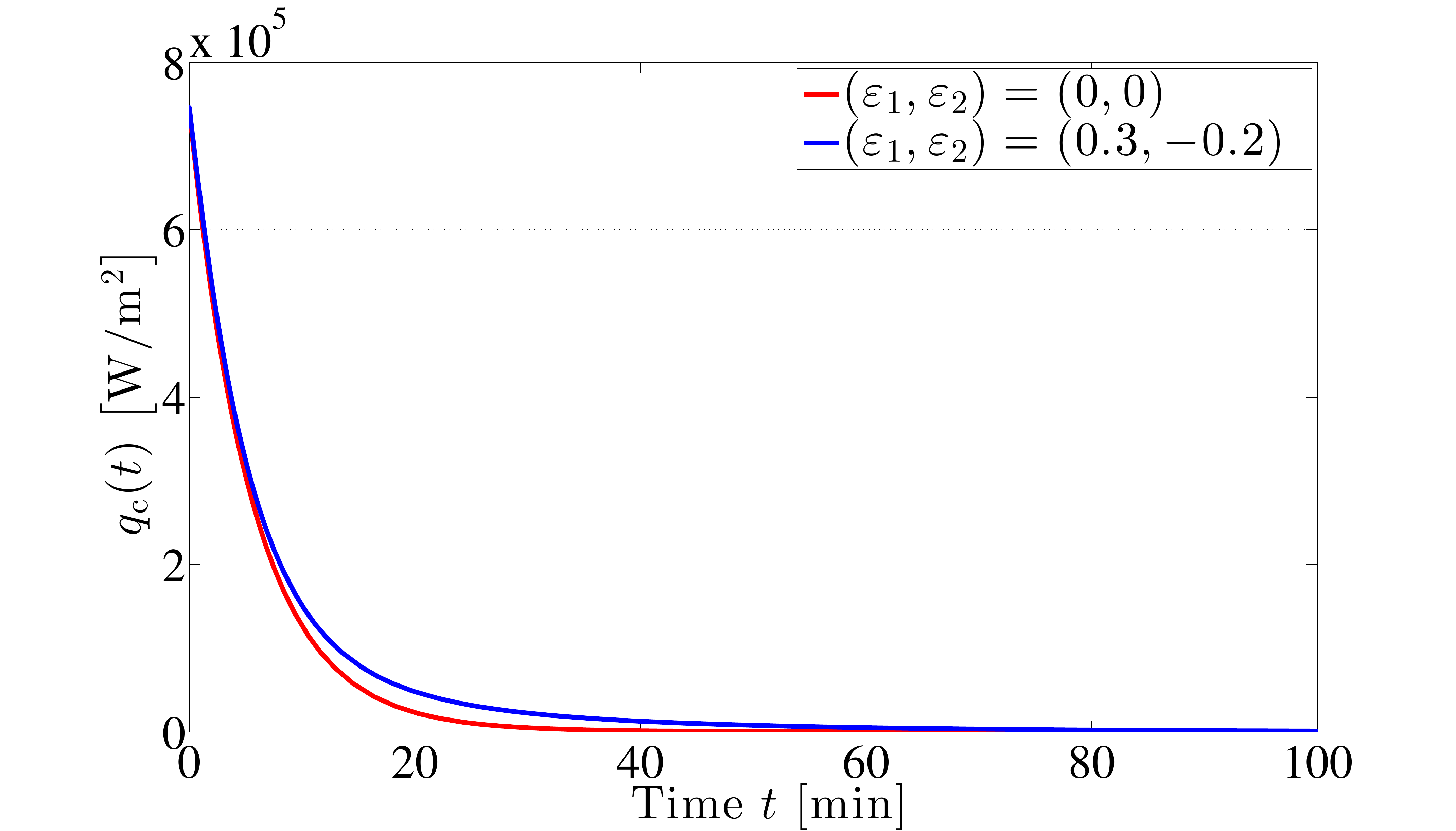}}
\subfloat[$T(s_0,t)$ warms up from $T_m$ and returns to it. ]{\includegraphics[width=3.0in]{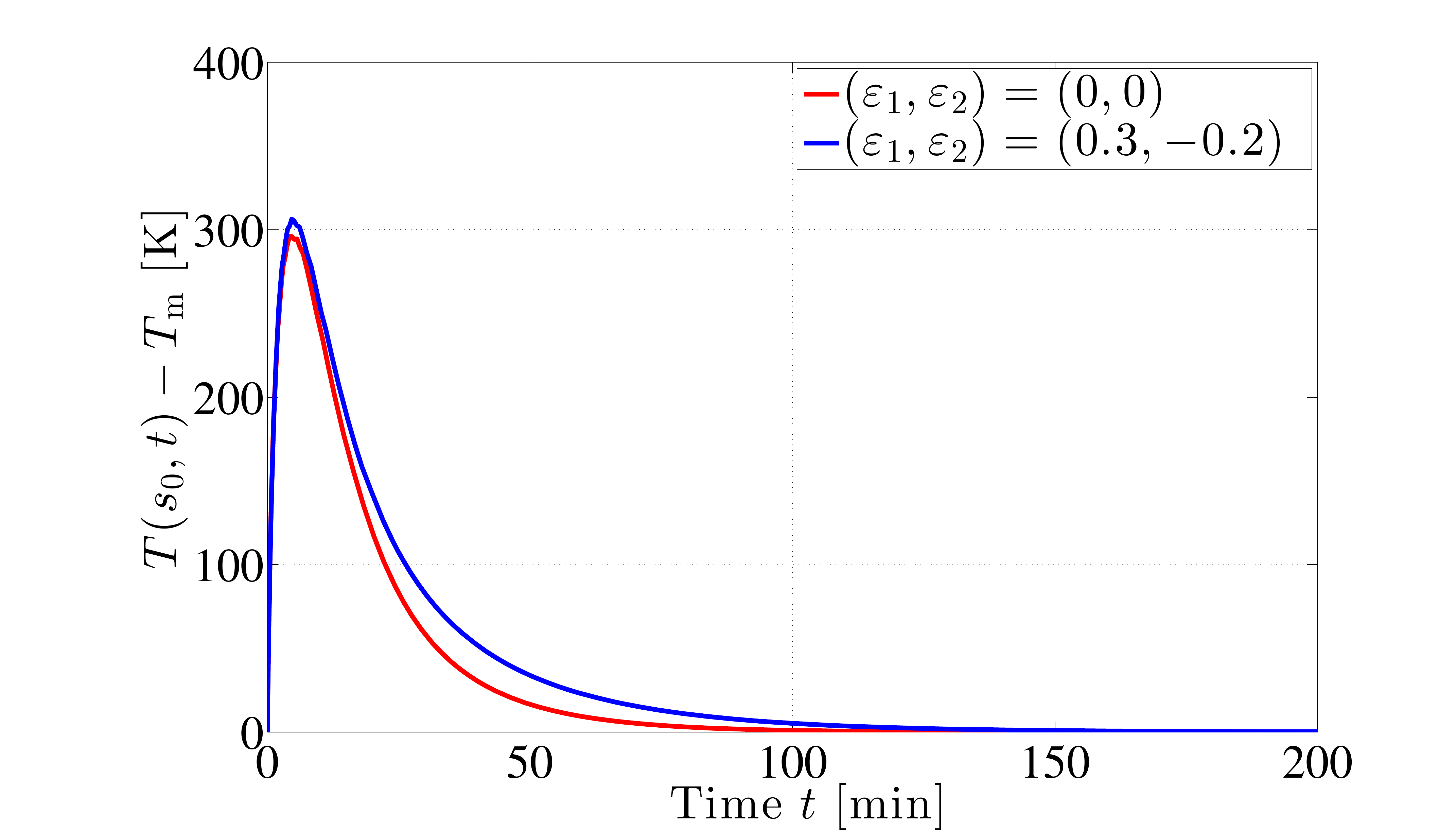}}\\
\caption{The closed-loop responses of the plant \eqref{eq:robustPDE}--\eqref{eq:robustODE} and the backstepping  control law  \eqref{Fullcontrol} with accurate parameters (red) and parameters perturbation (blue).}
\label{fig:Robust2}
\subfloat[$s(t)$ converges to $s_r$ without overshoot.]{\includegraphics[width=3.0in]{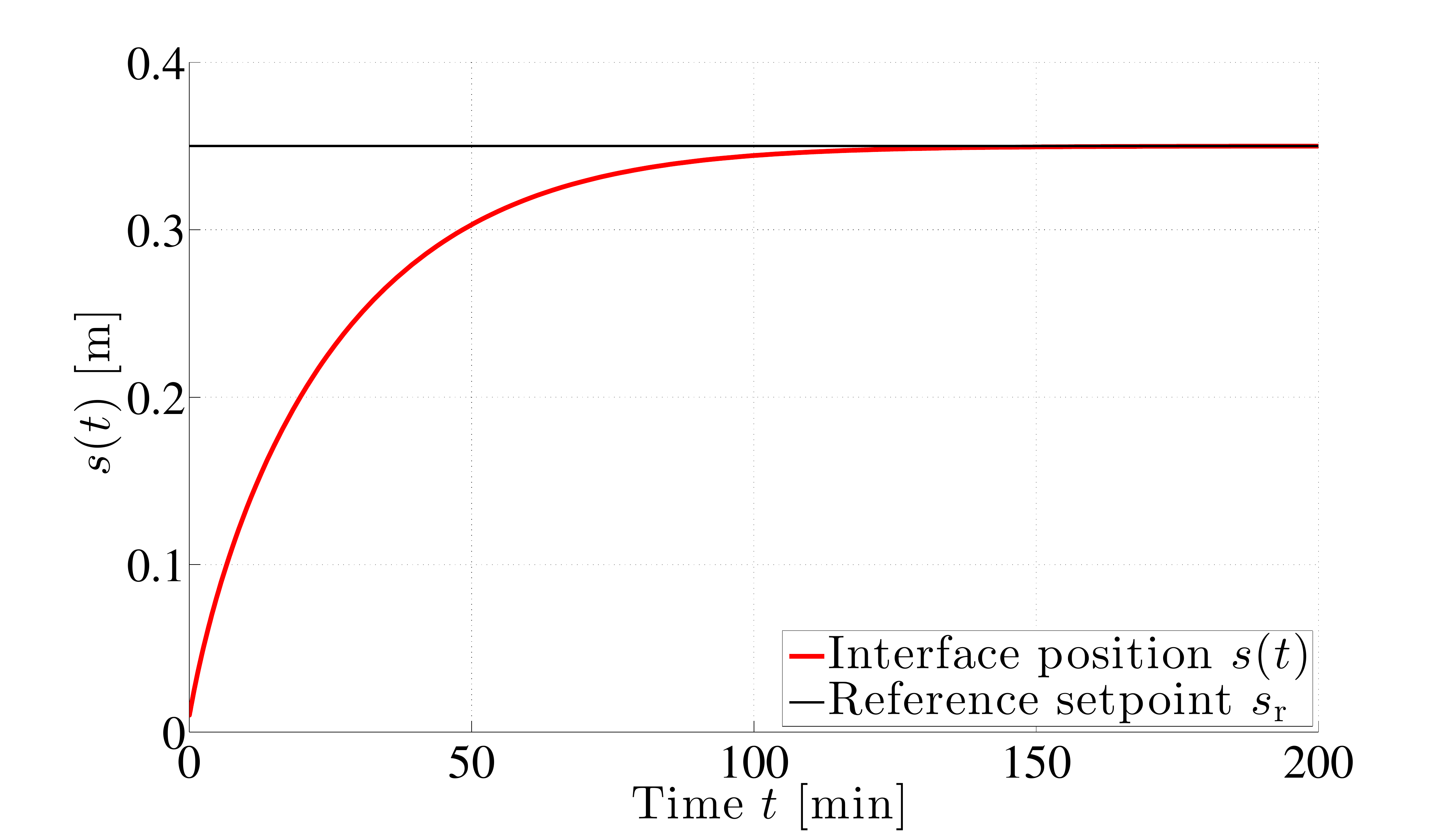}}
\subfloat[Positivity of the output feedback controller remains. ]{\includegraphics[width=3.0in]{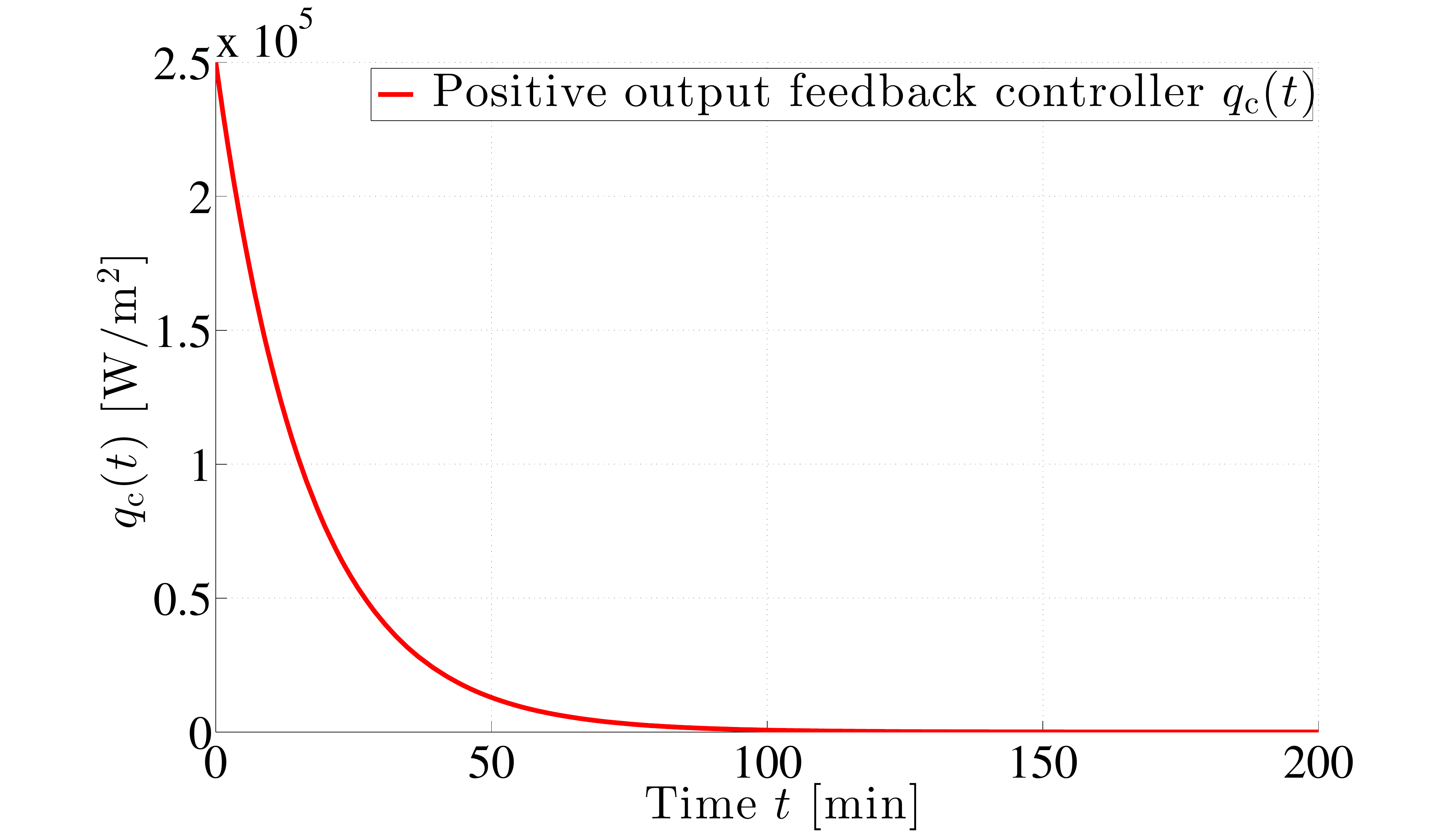}}\\
\subfloat[$T(s_0,t)$ warms up from $T_m$ and returns to it. ]{\includegraphics[width=3.0in]{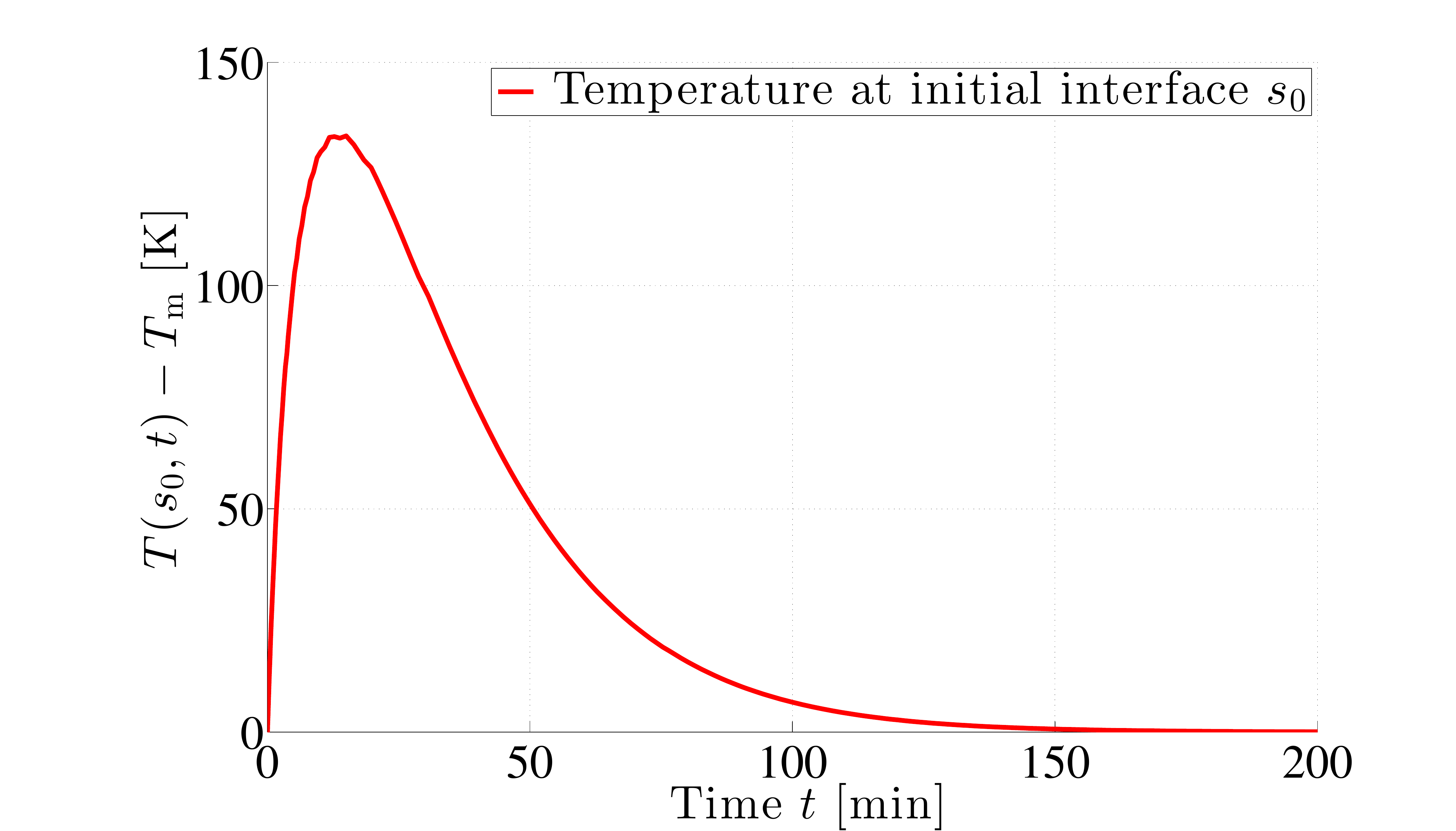}}
\subfloat[Negativity of the estimation errors remains. ]{\includegraphics[width=3.0in]{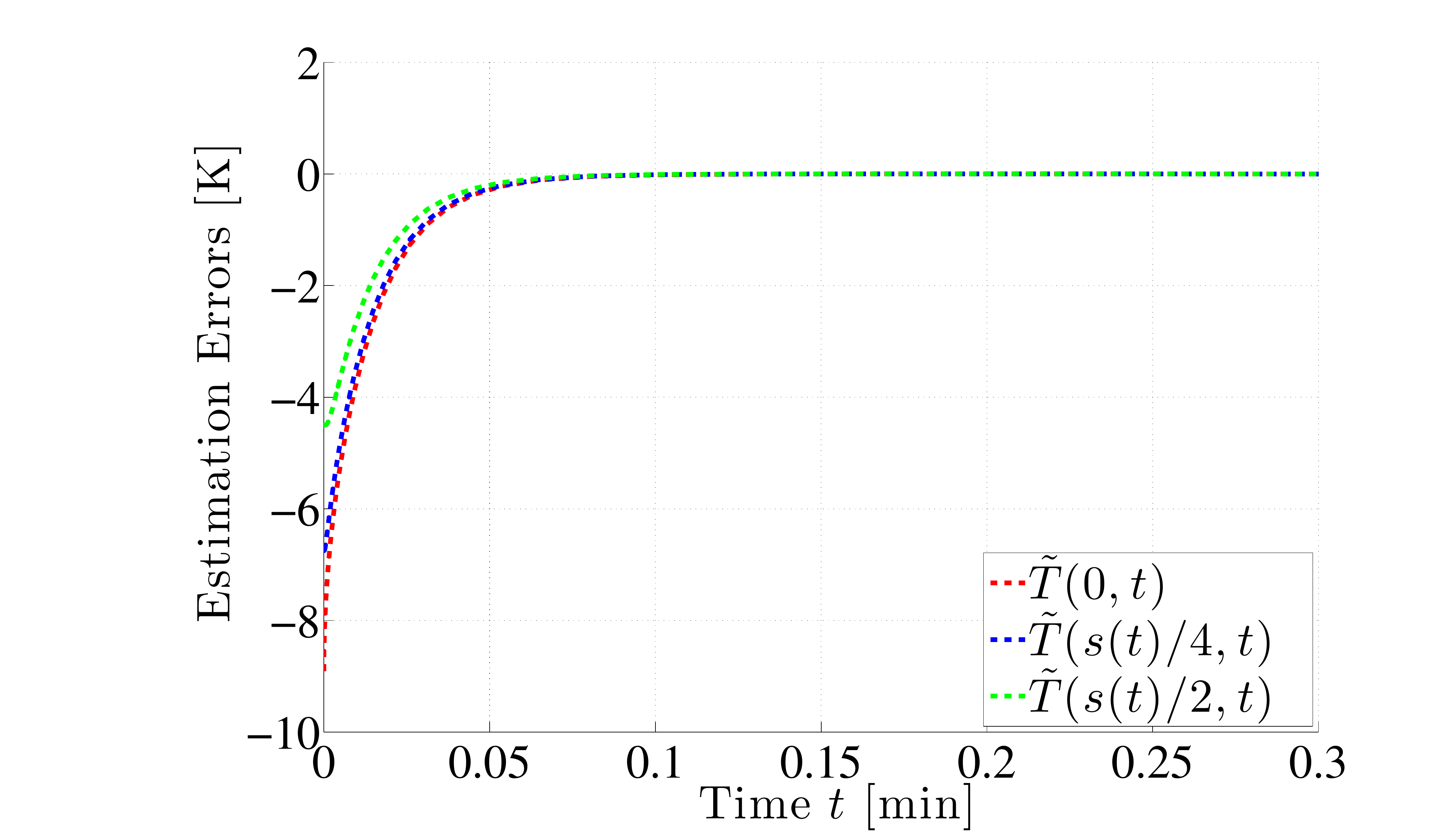}}
\caption{Simulation of the closed-loop system \eqref{eq:stefanPDE} - \eqref{eq:stefanODE} and the estimator  \eqref{observerPDE} - \eqref{P1x} with the output feedback control law \eqref{outctr}.}
\label{fig:simulation}
\end{figure*}

\section{Conclusion}\label{sec:conclusion}
This paper presented a control and estimation design for the one-phase Stefan problem via backstepping method. The system is a diffusion PDE on a moving interface domain which has the dynamics described by ODE. The novelties of this paper are summarized as followings. 
\begin{enumerate}
\item We developed a new approach for the global stabilization of nonlinear parabolic PDE via a nonlinear backstepping transformation. 
\item The closed-loop responses guarantee the physical constraints imposed by the validity of the model. 
\item A new formulation of the Lyapunov function for moving boundary PDE was applied and it showed the exponential stability of the closed loop system. 
\end{enumerate}
We emphasize that the proposed controller does not require the restriction imposed in \cite{maidi2014}  regarding the material properties. In that sense, our design can be applied to a wider  class of melting materials. Also, even if  the controller is same as the one proposed in  \cite{petrus10}, we ensure  the exponential stability of the sum of interface error and ${\cal H}_1$-norm of temperature error which is stronger than the  asymptotical stability result presented in \cite{petrus10}.

\appendices
\section{Backstepping Transformation for Moving Boundary}\label{appA}
\subsubsection{Direct Transformation}\label{appDBST}
Recall that the reference error system defined  in Section \ref{state-feedback} is 
\begin{align}\label{eq:appPDE}
u_{t}(x,t) =& \alpha u_{xx}(x,t), \quad 0\leq x\leq s(t)\\
 \label{eq:appBC1}u_x(0,t) =& -\frac{q_{{\mathrm c}}(t)}{k},\\
 \label{eq:appBC2}u(s(t),t) =&0,\\
\label{eq:appODE}\dot{X}(t) =&-\beta u_x(s(t),t).
\end{align}
Suppose that the backstepping transformation for moving boundary PDE is formulated as
\begin{align}\label{DBSTgene}
w(x,t)=&u(x,t)- \int_{x}^{s(t)} k(x,y)u(y,t) dy\notag\\
&-\phi(x-s(t)) X(t), 
\end{align}
which transforms the reference error system \eqref{eq:appPDE}-\eqref{eq:appODE} into the following target system 
\begin{align}\label{targenePDE}
w_t(x,t)=&\alpha w_{xx}(x,t)+\dot{s}(t)\phi'(x-s(t))X(t), \\
w_{x}(0,t) =& 0, \\
\label{targeneBC2}w(s(t),t) =& 0, \\
\label{targeneODE}\dot{X}(t) =& -cX(t) - \beta w_{x}(s(t),t). 
\end{align}
The stability of the target system \eqref{targenePDE}-\eqref{targeneODE} is guaranteed with the assumptions such that $\dot{s}(t)>0$, $s_0<s(t)<s_{{\mathrm r}}$ as shown in Appendix \ref{stabilitytarget}. 
Then, taking first and second spacial derivatives of \eqref{DBSTgene} and the first time derivative of \eqref{DBSTgene}, we obtain the following relation 
\begin{align}\label{kerPDE}
w_t(x,t)-&\alpha w_{xx}(x,t)-\dot{s}(t)\phi'(x-s(t))X(t)\notag\\
=&-\left(2\alpha \frac{d}{dx}k(x,x)\right)u(x,t)\nonumber\\
&-(\alpha k(x,s(t))-\beta\phi(x-s(t))) u_x(s(t),t)\notag\\
&+ \alpha \int_{x}^{s(t)} (k_{xx}(x,y)-k_{yy}(x,y))u(y,t) dx\nonumber\\
&+\alpha \phi''(x-s(t)) X(t).
\end{align}
Evaluating  \eqref{DBSTgene} and its spacial derivative at $x=s(t)$, we have 
\begin{align}\label{kerBC1}
w(s(t),t)=&-\phi(0) X(t), \\
\label{kerBC2}w_x(s(t),t)=&u_x(s(t),t)-\phi'(0) X(t). 
\end{align}
In order to hold \eqref{targenePDE}-\eqref{targeneODE} for any continuous functions $(u(x,t), X(t))$, by \eqref{kerBC1}-\eqref{kerPDE}, $k(x,y)$ and $\phi(x-s(t))$ satisfy 
\begin{align}
\phi''(x-s(t))=&0, \\
\phi(0)=&0, \quad \phi'(0) = \frac{c}{\beta}, \\
\frac{d}{dx}k(x,x) = &0, \quad k_{xx}(x,y)=k_{yy}(x,y), \\
k(x,s(t)) = &\frac{\beta}{\alpha}\phi(x-s(t)).
\end{align}
The solutions to the gain kernels are given by
\begin{align}
\phi(x-s(t))=&\frac{c}{\beta}(x-s(t)),\\
k(x,y)=&\frac{\beta}{\alpha} \phi(x-y).
\end{align}

\subsubsection{Inverse Transformation}\label{appIBST}
Suppose that the inverse transformation is formulated as
\begin{align}\label{eq:IBST2}
u(x,t) =& w(x,t) +\int_x^{s(t)} l(x,y)w(y,t) dy\notag\\
&+ \psi(x-s(t)) X(t),
\end{align}
where $l(x,y)$, $\psi(x-s(t))$ are the kernel functions. Taking derivative of \eqref{eq:IBST2} with respect to $t$ and $x$, respectively, along the solution of \eqref{eq:tarPDE}-\eqref{eq:tarODE}, the following relations are derived  
\begin{align}\label{eq:gradient}
u(s(t),t) =& \psi(0) X(t), \\
u_x(s(t),t) =& w_x(s(t),t) +\psi'(0)X(t),
\end{align}
\begin{align}\label{time-deriv}
&u_t(x,t) -\alpha u_{xx}(x,t) \notag\\
= & 2\alpha \left(\frac{d}{dx}l(x,x)\right)w(x,t)\notag\\
&-\alpha \int_x^{s(t)} (l_{xx}(x,y)-l_{yy}(x,y))w(y,t) dx\nonumber\\
&+\left(\alpha l(x,s(t)) - \beta\psi(x-s(t))\right)w_x(s(t),t)\notag\\
&- \left(c\psi(x-s(t))+\alpha \psi''(x-s(t))\right) X(t)\nonumber\\
&+\left\{\frac{c}{\beta}\left(1+\int_x^{s(t)} l(x,y) dx\right)-\psi'(x-s(t))\right\}\dot{s}(t) X(t).
\end{align}
In order to hold \eqref{eq:appPDE}-\eqref{eq:appODE} for any continuous functions $(w(x,t), X(t))$, by \eqref{eq:gradient}-\eqref{time-deriv}, $\psi(x-s(t))$ and $l(x,y)$ satisfy 
\begin{align}
\label{eq:hode}
\psi''(x-s(t)) =& -\frac{c}{\alpha} \psi(x-s(t))\\
\label{eq:hic1}
\psi(0) =& 0, \quad  \psi'(0) = \frac{c}{\beta},\\
\label{lcond1}\frac{d}{dx}l(x,x)=&0, \quad l_{xx}(x,y)=l_{yy}(x,y),\\
\label{lcond3} l(x,s(t)) =& \frac{\beta}{\alpha} \psi(x-s(t))\\
\label{addcond}\psi'(x-s(t))=& \frac{c}{\beta}\left(1+\int_x^{s(t)} l(x,y) dx\right).
 \end{align}
Using \eqref{eq:hode} and \eqref{eq:hic1}, the solution of $\psi(x-s(t))$ is given by
\begin{align}\label{hxy}
\psi(x-s(t)) = \frac{\sqrt{c\alpha}}{\beta} {\rm sin}\left( \sqrt{\frac{c}{\alpha}}(x-s(t))\right). 
\end{align}
The conditions \eqref{lcond1}-\eqref{lcond3} yields
\begin{align}\label{lxy}
l(x,y)=&\frac{\beta}{\alpha} \psi(x-y).
\end{align}
Then, the solutions \eqref{hxy} and \eqref{lxy} satisfy the condition \eqref{addcond} as well. 

\section{Stability Analysis}\label{stability}
In this section, we prove the exponential stability of $(w,X)$ system defined in \eqref{targenePDE}-\eqref{targeneODE} via Lyapunov analysis, which induces the stability  of the original $(u,X)$ system. The following assumptions on the interface dynamics
\begin{align}\label{assumdomain}
\dot{s}(t)>0, \quad 0<s_0<s(t)<s_{{\mathrm r}},
\end{align}
which are shown in Section \ref{positivness} are stated. 
\subsubsection{Stability of the Target System}\label{stabilitytarget}
Firstly, we show the exponential stability of the target system \eqref{targenePDE}-\eqref{targeneODE}. Consider the Lyapunov function $V$ such that 
\begin{align}\label{lyapgene}
V = &\frac{1}{2}||w||_{{\cal H}_1}^2+\frac{p}{2}X(t)^2, 
\end{align}
with a positive number $p>0$ to be chosen later. Then, by taking the time derivative of \eqref{lyapgene} along the solution of the target system \eqref{targenePDE}-\eqref{targeneODE}, we have
\begin{align}\label{timeder}
\dot{V}=&-\alpha|| w_{xx}||_{{\cal L}_2}^2 -\alpha ||w_{x}||_{{\cal L}_2}^2 \notag\\
&-pcX(t)^2 - p\beta X(t) w_x(s(t),t) \notag\\
&+\frac{\dot{s}(t)}{2} w_{x}(s(t),t)^2 + w_{t}(s(t),t) w_{x}(s(t),t) \notag\\
&-\dot{s}(t)\phi'(0) X(t)w_{x}(s(t),t) +\dot{s}(t) X(t)\left( \phi''(s(t)) w(0,t) \right. \notag\\
&\left.+ \int_0^{s(t)} f(x-s(t)) w(x,t)  dx\right), 
\end{align}
where $f(x) = \phi'(x) - \phi'''(x)$. The boundary condition \eqref{targeneBC2} yields 
\begin{align}\label{chain}
 w_{t}(s(t),t) =- \dot{s}(t) w_{x}(s(t),t),
\end{align}
 by chain rule $\frac{d}{dt} w(s(t),t) = w_{t}(s(t),t) + \dot{s}(t) w_{x}(s(t),t) = 0$. Substituting \eqref{chain} into \eqref{timeder}, we get
\begin{align}
\dot{V}=&-\alpha|| w_{xx}||_{{\cal L}_2}^2 -\alpha ||w_{x}||_{{\cal L}_2}^2\notag\\
&-pcX(t)^2- p\beta X(t) w_x(s(t),t) \notag\\
&+\dot{s}(t)\left( -\phi'(0) X(t)w_{x}(s(t),t)-\frac{1}{2}w_x(s(t),t)^2 \right)\notag\\
&+\dot{s}(t) X(t)\left( \phi''(s(t)) w(0,t) \right. \notag\\
&\left.+ \int_0^{s(t)} f(x-s(t)) w(x,t)  dx\right).
\end{align}
Define $m =\int_0^{s_{{\mathrm r}}} f(-x)^2 dx$. From the  assumption \eqref{assumdomain}, Young's, Causchy Schwartz, Pointcare's, Agmon's inequality, and choosing $p = \frac{c \alpha}{4 \beta^2 s_{{\mathrm r}}}$, we have
\begin{align}
\dot{V} \leq& - \frac{\alpha}{2}||w_{xx}||_{{\cal L}_2}^2 -\alpha || w_{x}||_{{\cal L}_2}^2 \notag\\
&-\frac{pc}{2}X(t)^2 +\dot{s}(t) \left\{ \frac{1+ \phi'(0)^2}{2} X(t)^2 \right.\notag\\
&\left. + 4 s_{{\mathrm r}} \phi''(s(t))^2 || w_{x}||_{{\cal L}_2}^2 + m || w||_{{\cal L}_2}^2 \right\},\notag\\
 \leq & - \frac{\alpha}{8s_{{\mathrm r}}^2} || w||_{{\cal H}_1}^2 -\frac{pc}{2}X(t)^2  \notag\\
 & +\dot{s}(t) \left\{4 s_{{\mathrm r}} \phi''(s(t))^2  ||w_{x}||_{{\cal L}_2}^2 \right.\notag\\
 &\left. + m || w||_{{\cal L}_2}^2 +  \frac{1+ \phi'(0)^2}{2} X(t)^2 \right\},\notag\\
 \leq & - bV + a \dot{s}(t)V.
\end{align}
where 
\begin{align}
a =& 2 {\rm max} \left\{ 4 s_{{\mathrm r}} \phi''(s(t))^2, m, \frac{1+ \phi'(0)^2}{2p}\right\},\\
b =& {\rm min} \left\{\frac{\alpha}{4s_{{\mathrm r}}^2}, c\right\}.
\end{align}

\subsubsection{Exponential stability for the original $(u, X)$-system}\label{stabilityoriginal}
The norm equivalence between the target system and original system is shown from the direct transformation \eqref{DBSTgene} and the inverse transformation \eqref{eq:IBST2}. Taking the square of the inverse transformation \eqref{eq:IBST2} and applying Young's and Cauchy Schwartz inequality, we have
\begin{align}\label{intnew}
u(x,t)^2 &\leq  3 w(x,t)^2 \notag\\
&+ 3\frac{\beta^2}{\alpha^2} \left( \int_x^{s(t)} \psi(x-y)^2 dy\right) \left( \int_x^{s(t)} w(y,t)^2 dy\right)\notag\\
&+ 3 \psi(x-s(t))^2 X(t)^2.
\end{align}
Integrating \eqref{intnew} from $0$ to $s(t)$ and applying Cauchy Schwartz inequality, we have
\begin{align}\label{app-u-1}
|| u||_{{\cal L}_2}^2  \leq&  3\left(1+\frac{\beta^2}{\alpha^2} s(t) \left( \int_{0}^{s(t)} \psi(-x)^2 dx\right)\right) ||  w ||_{{\cal L}_2}^2 \notag\\
&+ 3 \left( \int_{0}^{s(t)} \psi(-x)^2 dx \right) X(t)^2.
\end{align}
Taking the spatial derivative of \eqref{eq:IBST2} and by the same manner, we have 
\begin{align}
|| u_{x} ||_{{\cal L}_2}^2 &\leq 3 || w_{x} ||_{{\cal L}_2}^2\notag\\
&+  3\frac{\beta^2}{\alpha^2}  \left(\psi(0)^2+s(t) \left( \int_{0}^{s(t)} \psi'(-x)^2 dx\right)\right) || w ||_{{\cal L}_2}^2 \notag\\
&+ 3 \left( \int_{0}^{s(t)} \psi'(-x)^2 dx \right) X(t)^2. 
\end{align}
In this case, we have $\phi(0) = \psi(0) = 0$. Thus, all of the inequality are written as
\begin{align}\label{bound1}
||w||_{{\cal L}_2}^2 &\leq  M_1 ||u||_{{\cal L}_2}^2 +M_2 X(t)^2,\\
\label{bound2}||w_{x}||_{{\cal L}_2}^2 &\leq 3  ||u_{x}||_{{\cal L}_2}^2+ M_3 ||u||_{{\cal L}_2}^2 + M_4X(t)^2,\\
\label{bound3}||u||_{{\cal L}_2}^2  &\leq  M_5 ||w||_{{\cal L}_2}^2 + M_6 X(t)^2,\\
\label{bound4}||u_{x}||_{{\cal L}_2}^2 &\leq 3  ||w_{x}||_{{\cal L}_2}^2+  M_7 ||w||_{{\cal L}_2}^2 + M_8 X(t)^2, 
\end{align}
where $M_1 = 3\left(1+\frac{\beta^2}{\alpha^2} s_{{\mathrm r}} \left( \int_{0}^{s_{{\mathrm r}}} \phi(-x)^2 dx\right)\right)$, $M_2 =  3 \left( \int_{0}^{s_{{\mathrm r}}} \phi(-x)^2 dx \right)$, $M_3 =  3\frac{\beta^2}{\alpha^2} s_{{\mathrm r}} \left( \int_{0}^{s_{{\mathrm r}}} \phi'(-x)^2 dx\right)$, $M_4 = 3 \left( \int_{0}^{s_{{\mathrm r}}} \phi'(-x)^2 dx \right)$, $M_5 = 3\left(1+\frac{\beta^2}{\alpha^2} s_{{\mathrm r}} \left( \int_{0}^{s_{{\mathrm r}}} \psi(-x)^2 dx\right)\right)$, $M_6 = 3 \left( \int_{0}^{s_{{\mathrm r}}} \psi(-x)^2 dx \right)$, $M_7 = 3\frac{\beta^2}{\alpha^2} s_{{\mathrm r}} \left( \int_{0}^{s_{{\mathrm r}}} \psi'(-x)^2 dx\right)$, $M_8 = 3 \left( \int_{0}^{s_{{\mathrm r}}} \psi'(-x)^2 dx \right)$.
Adding \eqref{bound1} to \eqref{bound2} and \eqref{bound3} to \eqref{bound4}, we derive the following inequality
\begin{align}\label{eq:bound}
\underline{\delta}\left(||u||_{{\cal H}_1}^2+X(t)^2\right)&\leq||w||_{{\cal H}_1}^2 +pX(t)^2\notag\\
&\leq \bar{\delta}\left(||u||_{{\cal H}_1}^2+X(t)^2\right),
\end{align}
where $\bar{\delta}={\rm max}\{M_1+M_3, p+M_2+M_4\}$, $\underline{\delta}=\frac{{\rm min}\left\{1,p\right\}}{{\rm max}\left\{M_5+M_7, M_6+M_8+1\right\}}$. Define a parameter $D>0$ as $D=\frac{\bar{\delta}}{\underline{\delta}}e^{as_{{\mathrm r}}}$. Then, with the help of \eqref{eq:bound}, \eqref{expstabilityw}, we deduce that there exists $D>0$ and $b>0$ such that
\begin{align}
||u||_{{\cal H}_1}^2+X(t)^2 \leq D \left(||u_0||_{{\cal H}_1}^2+X(0)^2\right) e^{-bt}.
\end{align}

\section*{Acknowledgment}

The authors would like to thank S. Tang for her help to ensure the first result of Theorem \ref{Theo-1}.

\ifCLASSOPTIONcaptionsoff
  \newpage
\fi


\begin{thebibliography}{99}

\bibitem{Armaou01}
A. Armaou and P.D. Christofides,
\newblock ``Robust control of parabolic PDE systems with time-dependent spatial domains,''
\newblock {\em Automatica}, vol. 37, pp. 61--69,
  2001.
  
     \bibitem{Baccoli2015}
 A. Baccoli, A. Pisano, Y. Orlov,
    \newblock ``Boundary control of coupled reaction-advection-diffusion systems with spatially-varying coefficients,"
\newblock {\em Automatica}, vol. 54, pp. 80--90, 2015.
  
 \bibitem{bekiaris2013compensation}
N. Bekiaris-Liberis and M. Krstic, 
\newblock ``Compensation of state-dependent input delay for nonlinear systems,''
\newblock {\em IEEE Transactions on Automatic Control}, vol. 58, no. 2, pp. 275--289, 
Feb. 2013.

 \bibitem{bekiaris2014wave}
N. Bekiaris-Liberis and M. Krstic, 
\newblock ``Compensation of wave actuator dynamics for nonlinear systems,''
\newblock {\em IEEE Transactions on Automatic Control}, vol. 59(6), pp.1555-1570, 2014.

\bibitem{Boon2014}
N. Boonkumkrong and S. Kuntanapreeda, 
\newblock ``Backstepping boundary control: An application to rod temperature control with Neumann boundary condition,"
\newblock {\em Proceedings of the Institution of Mechanical Engineers, Part I: Journal of Systems and Control Engineering}, vol. 228(5), pp.295-302, 2014.

\bibitem{Dejan2001}
D.  M. Boskovic, M. Krstic, and W. Liu,
\newblock ``Boundary control of an unstable heat equation via
measurement of domain-averaged temperature,"
\newblock {\em IEEE Transactions on Automatic Control}, vol. 46(12), pp. 2022--2028, 2001.

\bibitem{Cai2015}
X. Cai and M. Krstic, 
\newblock ``Nonlinear control under wave actuator dynamics with time- and state-dependent moving boundary,"
\newblock {\em International Journal of Robust and Nonlinear Control}, vol. 25(2), pp.222-251, 2015.


\bibitem{Christofides98_Parabolic}
P.D. Christofides,
\newblock ``Robust control of parabolic {PDE} systems,''
\newblock {\em Chemical Engineering Science}, vol. 53, pp. 2949--2965, 
1998.

\bibitem{conrad_90}
F. Conrad, D. Hilhorst, and T.I. Seidman,
\newblock ``Well-posedness of a moving boundary problem arising in a
  dissolution-growth process,''
\newblock {\em Nonlinear Analysis}, vol. 15, pp. 445--465, 1990.

\bibitem{Daraoui2010}
N.~Daraoui, P.~Dufour, H.~Hammouri, and A.~Hottot,
\newblock ``Model predictive control during the primary drying stage of
  lyophilisation,"
\newblock {\em Control Engineering Practice}, vol. 18, pp. 483--494, 2010.

  
  \bibitem{Deut2016}
 J. Deutscher
     \newblock ``Backstepping design of robust output feedback regulators for boundary controlled parabolic PDEs,"
\newblock {\em IEEE Transactions on Automatic Control}, vol. 61(8), pp. 2288-2294, 2016.


\bibitem{Friedman1999}
A. Friedman and F. Reitich, 
\newblock ``Analysis of a mathematical model for the growth of tumors,"
\newblock  {\em Journal of mathematical biology}, vol. 38(3), pp.262-284, 1999.


\bibitem{Gupta03}
S.~Gupta,
\newblock {\em The Classical Stefan Problem. Basic Concepts, Modelling and
  Analysis}.
\newblock North-Holland: Applied mathematics and Mechanics, 2003.

\bibitem{Izadi15}
M.~Izadi and S.~Dubljevic,
\newblock ``Backstepping output-feedback control of moving boundary parabolic
  {PDE}s,"
\newblock {\em European Journal of Control}, vol. 21, pp. 27 -- 35, 2015.

\bibitem{Shumon16}
S.~Koga, M.~Diagne, S.~Tang, and M.~Krstic,
\newblock ``Backstepping control of the one-phase Stefan problem,"
\newblock In {\em 2016 American Control Conference (ACC)}, pages 2548--2553.
  IEEE, 2016.

\bibitem{Shumon16CDC}
S. Koga, M. Diagne, and M. Krstic,
\newblock ``Output feedback control of the one-phase Stefan problem,"
\newblock In {\em 55th Conference on Decision and Control (CDC)}, pages 526--531.
  IEEE, 2016.


\bibitem{krstic09}
M.~Krstic,
\newblock ``Compensating actuator and sensor dynamics governed by diffusion
  {PDE}s,"
\newblock {\em Systems \& Control Letters}, vol. 58, pp. 372--377, 2009.

\bibitem{krstic2010}
M. Krstic, 
\newblock ``Input delay compensation for forward complete and strict-feedforward nonlinear systems,"
\newblock {\em IEEE Transactions on Automatic Control}, vol. 55(2), pp.287-303, 2010.

\bibitem{krstic2008backstepping}
M.~Krstic and A.~Smyshlyaev,
\newblock ``Backstepping boundary control for first-order hyperbolic {PDE}s and
  application to systems with actuator and sensor delays,"
\newblock {\em Systems \& Control Letters}, vol. 57, pp. 750--758, 2008.

\bibitem{krstic2008boundary}
M.~Krstic and A.~Smyshlyaev,
\newblock {\em Boundary Control of {PDE}s: A Course on Backstepping Designs}.
\newblock Singapore: SIAM, 2008.

  \bibitem{Krstic2008adap}
 M. Krstic and A. Smyshlyaev,
 \newblock `` Adaptive boundary control for unstable parabolic PDEs--Part I: Lyapunov Design,"
\newblock {\em IEEE Transactions on Automatic Control}, vol. 53(7), pp. 1575--1591, 2008.


\bibitem{Lei2013}
C. Lei, Z. Lin, and H. Wang, 
\newblock ``The free boundary problem describing information diffusion in online social networks,"
\newblock {\em Journal of Differential Equations}, vol. 254(3), pp.1326-1341, 2013.

\bibitem{maidi2014}
A.~Maidi and J.-P. Corriou,
\newblock ``Boundary geometric control of a linear stefan problem,"
\newblock {\em Journal of Process Control}, vol. 24, pp. 939--946, 2014.

\bibitem{maykut71}
G.A. Maykut and N. Untersteiner,
\newblock ``Some results from a time dependent thermodynamic model of sea ice,"
\newblock {\em Journal of Geophysical Research}, vol. 76, pp. 1550--1575, 1971.

\bibitem{nonlinearPDE}
C.V. Pao,
\newblock {\em Nonlinear Parabolic and Elliptic Equations}.
\newblock Springer, 1992.

\bibitem{Petit10}
N.~Petit,
\newblock ``Control problems for one-dimensional fluids and reactive fluids with
  moving interfaces,"
\newblock In {\em Advances in the theory of control, signals and systems with
  physical modeling}, volume 407 of {\em Lecture notes in control and
  information sciences}, pages 323--337, Lausanne, Dec 2010.
  
  \bibitem{petrus12}
B. Petrus, J. Bentsman, and B.G. Thomas,
\newblock ``Enthalpy-based feedback control algorithms for the Stefan problem,"
\newblock {\em Decision and Control (CDC), 2012 IEEE 51st Annual Conference on}, pp. 7037--7042, 2012.

  \bibitem{petrus10}
B. Petrus, J. Bentsman, and B.G. Thomas,
\newblock ``Feedback control of the two-phase Stefan problem, with an application to the continuous casting of steel,"
\newblock {\em Decision and Control (CDC), 2010 49th IEEE Conference on}, pp. 1731--1736, 2010.

\bibitem{Smyshlyaev2004}
A. Smyshlyaev and M. Krstic,
\newblock ``Closed-form boundary State feedbacks for a class of 1-D partial integro-differential equations,"
\newblock {\em IEEE Transactions on Automatic Control}, vol. 49(12), pp. 2185--2202, 2004.

\bibitem{Smyshlyaev2005}
A. Smyshlyaev and M. Krstic,
\newblock `` Backstepping observers for a class of parabolic PDEs,"
\newblock {\em Systems \& Control Letters}, vol. 54(7), pp. 613--625, 2005.

   \bibitem{Smyshlyaev2005b}
A. Smyshlyaev and M. Krstic,
\newblock ``On control design for PDEs with space-dependent diffusivity or time-dependent reactivity,"
\newblock {\em Automatica}, vol. 41(9), pp. 1601--1608, 2005.

    
\bibitem{Krstic2008adap2}
  A. Smyshlyaev and M. Krstic,
 \newblock ``Adaptive boundary control for unstable parabolic PDEs--Part II: Estimation-based designs,"
\newblock {\em Automatica}, vol. 43(9), pp.  1543--1556, 2007.

  
  \bibitem{Krstic2008adap3}
  A. Smyshlyaev and M. Krstic,
 \newblock ``Adaptive boundary control for unstable parabolic PDEs--Part III: Output
feedback examples with swapping identifiers,"
\newblock {\em Automatica}, vol. 43(9), pp.   1557--1564, 2007.



\bibitem{srinivasan04}
V. Srinivasan and J. Newman,
\newblock ``Discharge model for the lithium iron-phosphate electrode,"
\newblock {\em Journal of the Electrochemical Society}, vol. 151, pp. A1517--A1529, 2004.


\bibitem{stefan91}
J. Stefan,
\newblock ``Uber die Theorie der Eisbildung, insbesondere uber die Eisbildung im Polarmeere,"
\newblock {\em Annalen der Physik}, vol. 278, pp. 269--286, 1891.

\bibitem{susto10}
G.A. Susto and M. Krstic,
\newblock ``Control of {PDE}--{ODE} cascades with Neumann interconnections,"
\newblock {\em Journal of the Franklin Institute}, vol. 347, pp. 284--314, 2010.

\bibitem{tang11}
S. Tang and C. Xie,
\newblock ``State and output feedback boundary control for a coupled {PDE}--{ODE} system,"
\newblock {\em Systems \& Control Letters}, vol. 60, pp. 540--545, 2011.

  \bibitem{Vazquez2017}
  R. Vazquez and M. Krstic,
    \newblock ``Boundary control and estimation of reaction-diffusion equations on the sphere under revolution symmetry conditions,"
\newblock {\em International Journal of Control}, to appear.
  
  
   \bibitem{Vazquez2017b}
  R. Vazquez and M. Krstic,
    \newblock ``Boundary control of coupled reaction-advection-diffusion systems with spatially-varying coefficients,"
\newblock {\em IEEE Transactions on Automatic Control}, to appear.
  

\bibitem{wettlaufer91}
J.S. Wettlaufer,
\newblock ``Heat flux at the ice-ocean interface,"
\newblock {\em Journal of Geophysical Research: Oceans}, vol. 96, pp. 7215--7236, 1991.

\bibitem{zalba03}
B. Zalba, J.M. Marin, L.F. Cabeza, and H. Mehling,
\newblock ``Review on thermal energy storage with phase change: materials, heat transfer analysis and applications,"
\newblock {\em Applied thermal engineering}, vol. 23, pp. 251--283, 2003.





  

  


  
  
  
  




  




     
\end{thebibliography}
\end{document}